\documentclass[a4paper,reqno,10pt]{amsart}

\usepackage[all]{xy}
\usepackage[english]{babel}
\usepackage{amssymb}
\usepackage{graphicx}
\usepackage{amsmath,amsthm,amssymb,xypic,verbatim}

\newcommand{\p}{\mathbb{P}} 
\newcommand{\pp}{\mathbb{P}}

\newcommand{\N}{\mathbb{N}}
\newcommand{\Z}{\mathbb{Z}}
\newcommand{\C}{\mathbb{C}}

\newcommand{\I}{\mathcal{I}}
\newcommand{\E}{\mathcal{E}}

\newcommand{\X}{\mathcal{X}}
 
\newcommand{\f}{\varphi}
\newcommand{\A}{\mathbb{A}}

\newcommand{\T}{\mathbb{T}}

\newcommand{\LL}{\mathcal{L}}
\newcommand{\LD}{\mathbb{L}}
\newcommand{\M}{\mathcal{M}}
\newcommand{\oo}{\mathcal{O}}
\newcommand{\mul}{\mbox{mult}}
\newcommand{\Bs}{\mathop{\rm Bs}\nolimits}

\newcommand{\Sec}{\mathop{{\rm{\mathbb S}ec}}\nolimits}
\newcommand{\Span}[1]{\langle#1\rangle}
\newcommand{\expdim}{\mathop{\rm expdim}\nolimits}
\newcommand{\virtdim}{\mathop{\rm virtdim}\nolimits}
\newcommand{\h}{\mbox{h}}

\theoremstyle{plain}

\newtheorem{thm}{Theorem}
\newtheorem{pro}[thm]{Proposition}
\newtheorem{lem}[thm]{Lemma}   
\newtheorem{cor}[thm]{Corollary}
\newtheorem{claim}{Claim}
\theoremstyle{definition}
\newtheorem{construction}[thm]{Construction}
\newtheorem{notaz}[thm]{Notation}

\newtheorem{es}[thm]{Example}
\newtheorem{con}[thm]{Conjecture}

\newtheorem{dfn}[thm]{Definition}
\newtheorem{rmk}[thm]{Remark}
\newtheorem{remark}[thm]{Remark}

\newcommand{\mult}[0]{\operatorname{mult}}

\begin{document}

\title[Identifiability and Cremona transformations]{Identifiability of homogeneous polynomials and Cremona transformations}

\author{Francesco Galuppi}
\author{Massimiliano Mella}
\address{
Dipartimento di Matematica e Informatica\ 
Universit\`a di Ferrara\\
Via Machiavelli 35\\
44100 Ferrara Italia}
\email{glpfnc@unife.it, mll@unife.it}
\date{September 2017}
\thanks{Partially supported by Progetto MIUR ``Geometria
  sulle variet\`a algebriche''}
\maketitle	
\begin{abstract}
  A homogeneous polynomial of degree $d$ in $n+1$ variables is
  identifiable if it admits a unique additive  
  decomposition in powers of linear forms. Identifiability is expected
  to be very rare. In this paper we conclude a work started more than
  a century ago and we describe all values of $d$ and $n$ for which a
  general polynomial of degree $d$ in $n+1$ variables is
  identifiable.  This is done by classifying a special class
  of Cremona transformations of projective spaces.
\end{abstract}
\section*{Introduction}
In this paper we are interested in additive decompositions of homogeneous polynomials. This is usually called Waring problem, after 1770 Waring statement on additive decomposition of integers in additive powers.
 Let $F\in \C[x_0,\ldots, x_n]_d$ be a general homogeneous polynomial of degree $d$. The additive decomposition we are looking for is

$$F=L_1^d+\ldots+L_h^d ,$$
where $L_i\in \C[x_0,\ldots, x_n]_1$ are linear forms. The problem is
classical. The first results are due to Sylvester \cite{Sy}, Hilbert
\cite{Hi}, Richmond \cite{Ri}, and Palatini \cite{Pa}, among
others. The problem addressed was to find special values of $d$ and
$n$ for which the decomposition is unique. When this happens the
decomposition yields a canonical form of a general polynomial. In
modern terminology this is called $h$-identifiability. Applications of
identifiability range from Blind Signal Separation to Phylogenetic and Algebraic statistic, see \cite{Lan} for an account of applications, and it is studied for both polynomials and tensors, see for instance \cite{COV1} \cite{AHJKS}. It is
expected that canonical forms are a  rare phenomenon, see conjectures in 
\cite{OS}.
The following is the list of known identifiable cases. A general form $f$ of degree $d$ in $n + 1$ variables is $s$-identifiable in the following cases:
\begin{itemize}
\item[-] $n = 1$, $d = 2k - 1$ and $s = k$, \cite{Sy}
\item[-] $ n = 3$, $d = 3$ and $s = 5$ Sylvester's Pentahedral Theorem,
  \cite{Sy}
\item[-] $n = 2$, $d = 5$ and $s = 7$ \cite{Hi}, \cite{Ri}, \cite{Pa},
  see also \cite{MM} for a different approach.
\end{itemize}
All these cases where known a century ago and were expected to be the only ones, \cite{Br}. Only quite recently 
a significant result was obtained in \cite{Me1} where the
author proved that these are the only cases if either $d>n$ or $n\leq
3$.

In this paper we complete the study of identifiability proving that these are the only identifiable polynomials.
\begin{thm}\label{th:1} Let $f$ be a general homogeneous form of degree $d$ in
  $n+1$ variables. Then $f$ is $s$-identifiable if and
  only if $(n,d,s)=(1,2k-1,k),(3,3,5),(2,5,7)$.
\end{thm}

Despite its algebraic statement we approach the problem from a
birational geometry point of view. The starting point is
\cite[Theorem 2.1]{Me1} where it is proved that identifiability forces a
particular tangential projection of the Veronese variety to be birational. This
projection is associated to linear systems with imposed singularities. 
Then our main result is a consequence of the following statement about
Cremona modifications of $\p^n$, which is of interest in itself.
\begin{thm}
  \label{thm:Cremona-identifiability} Let $\LL_{n,d}(2^h)\subset|\oo_{\p^n}(d)|$ be the
  linear system of forms of degree $d$ with $h$ double points in general position and $\f_{n,d,h}$
  the rational map associated to it. Then $\f_{n,d,h}$ is a Cremona
  transformation, i.e. $\dim\LL_{n,d}(2^h)=n$ and $\f_{n,d,h}:\p^n\dasharrow\p^n$ is birational, if and only if
  \begin{itemize}
  \item[-]  $n=1$, $d=2k+1$,  and $h=k$,
\item[-] $n=2$, $d=5$ and $h=6$,
\item[-] $n=3$, $d=3$ and $h=4$. 
  \end{itemize}
\end{thm}
The main difficulty in proving
Theorem~\ref{thm:Cremona-identifiability} is to control the
singularities and the base locus of the linear system
$\LL_{n,d}(2^h)$. The first task is accomplished in  \cite[Corollary
4.5]{Me1}.  It is proved that, for $d\geq 4$,  the singularities of
$\LL_{n,d}(2^h)$ are only the double points imposed, the degree $3$
case has been recently completed by \cite{COV2}. This allowed to
conclude in the mentioned range. Unfortunately if $d\leq n$ it is
necessary to control not only the singularities but also the base
locus of these linear systems to bound the degree of the map. In
\cite{Me2} some special cases were proved assuming a divisibility
condition on the degree. Here we approach the problem from a
different perspective. Instead of trying to bound directly the degree
of the map associated to $\LL_{n,d}(2^h)$ we produce a degeneration
of the imposed singularities in such a way that the limit linear
system admits a hyperplane on which the 
restricted map is still expected to be non birational and then proceed
by induction trying to bound the sectional genus of the linear systems we are considering.  

This reminds the techniques of interpolation. Indeed the linear systems we are interested in have been studied for the
interpolation problem and we profit both of 
Alexander--Hirschowitz' paper \cite{AH} and of more recent approaches due to Postinghel \cite{Pos} and
Brambilla--Ottaviani \cite{BO}. The proof of
Theorem~\ref{thm:Cremona-identifiability} is done by induction on $n$.
The induction step is done via a careful choice of numbers. This
numerology is the core of the (differentiable) Horace method in
\cite{AH}, where it is played a double induction on both degree and
dimension. We were not able to control the sectional genus along these specializations. For this reason we have to develop a  different
approach based only on dimension induction. We let some double points
collapse into a 3-ple point with tangent directions. The latter
allows us to make induction work and to
study the restriction of the linear system to a hyperplane. 

This
leads us to study the standard interpolation 
problems for linear systems with one triple point, with tangent
directions in general position, and a bunch of double points. 
The first step of induction is the study of planar linear
systems. Here we benefit from the theory developed around
Harbourne--Hirschowitz conjecture, see \cite{Cil}, in 
particular we use the results in \cite{CM} about quasi-homogeneous multiplicity.
As usual in interpolation problems, for low degrees $d\leq 5$ and in
particular for cubics, we need special arguments. Once we worked out the case with tangent direction in general position we extend it to the set of tangent direction arising in the flat limit of $(n+1)$ double points that collapse to a point of multiplicity 3 concluding the proof.

Here is an outline of the paper. 
In section~\ref{sec:not} we introduce the notation we use and describe the reduction from Theorem~\ref{th:1} to Theorem~\ref{thm:Cremona-identifiability}. In section~\ref{sec:limits} we prove the result we need about degeneration with and without collapsing points. In particular we compute the limits of some special
configurations of double points. We were not able to find the theorems
we needed stated explicitly, therefore we proved everything but this section is
inspired by Nesci's and Postinghel's theses, \cite{Ne} \cite{Posth}.
In section~\ref{sec:induction} we prove the main induction
argument. This is done in several steps to help, at least this is our
hope, the reader to digest all the numerology needed. 
In section~\ref{sec:plane} we bound the sectional genus of the linear
system we are interested in.
In section~\ref{sec:special} we apply the result obtained to the
flat limit of $(n+1)$ collapsing double points, settle the special
case of cubics and prove Theorem~\ref{thm:Cremona-identifiability}.

{\sl Acknowledgments} We are grateful to the referees for a careful
reading, for suggestions that improves the expositions,  and for pointing out wrong computations in the first
version of the paper.

\section{Notations and preliminaries} \label{sec:not}
We work over the complex field.
	\begin{dfn} Let $Z$ be a $0$-dimensional scheme. The {degree}, or {length}, of $Z$, denoted by $\deg Z$, is the dimension of its ring of regular functions as a complex vector space.
	\end{dfn}
We start recalling some useful facts about 0-dimensional schemes.
	\begin{pro}\label{stessogrado} \label{basic} Let $Z$ be a $0$-dimensional subscheme. 
          \begin{itemize}
          \item[i)]The degree of $Z$ is the value of the Hilbert
            polynomial of $Z$.
\item[ii)]Let $X$, $Y$ be $0$-dimensional schemes such that $X\subseteq Y$ and $\deg X = \deg Y$. Then $X = Y$.
          \end{itemize}
\end{pro}
We deal with linear systems on $\p^n$ with assigned singularities and tangent directions.
	\begin{notaz} Let $\{p_1,\ldots, p_r\}\subset\p^n$ be a set of points and $\{q_{1},\ldots,q_{j}\}\in \p(\T_{p_1}\p^n)$ a set of tangent directions (infinitely near points) in $p_1$.
The linear system 
$$\LL_{n,d}(m_1,\dots,m_r)(p_1[\{q_{1},\ldots,q_{j}\}],p_2,\ldots,p_r)\subset|\oo_{\p^n}(d)|$$ 
is the sublinear system of hypersurfaces having multiplicities $m_i$ at the point $p_i$ and whose tangent cone at $p_1$ contains $\{q_1,\ldots, q_j\}$.
If either the points $\{p_1,\ldots, p_r\}$ and  $\{q_{1},\ldots,q_{j}\}$ are in general position or no confusion is likely to arise, we indicate
$$\LL_{n,d}(m_1[j],m_2\dots,m_r):=\LL_{n,d}(m_1,\dots,m_r)(p_1[\{q_{1},\ldots,q_{j}\}],p_2,\ldots,p_r)$$
and
$$\LL_{n,d}(m_1,\dots,m_r):=\LL_{n,d}(m_1[0],m_2\dots,m_r).$$
Moreover if $m_1=\ldots=m_g=m$ we indicate
$$\LL_{n,d}(m^g,m_{g+1},\dots,m_r):=\LL_{n,d}(m_1,\dots,m_r).$$
	\end{notaz}
	
	\begin{dfn}
The {virtual dimension} of such a linear system is 

	$$\virtdim\LL_{n,d}(m_1[j],m_2\dots,m_r)=\binom{d+n}{n}-1-\sum_{i=1}^r\binom{m_i-1+n}{n}-j,$$

the {expected dimension}  is defined
as			$$\expdim\LL_{n,d}(m_1[j],\dots,m_r)=\max
\left\{\virtdim\LL_{n,d}(m_1[j],m_2\dots,m_r),-1\right\},$$
where dimension $-1$ indicates that the linear system is expected to be empty.
Note that 
$$\dim\LL_{n,d}(m_1[j],m_2\dots,m_r)\ge\expdim
\LL_{n,d}(m_1[j],m_2\dots,m_r).$$
If $\dim\LL_{n,d}(m_1[j],m_2\dots,m_r)>\dim\LL_{n,d}(m_1[j],m_2\dots,m_r)$, then the linear system is said to be {special}. Otherwise it is called {nonspecial}.
\end{dfn}

The speciality of linear systems has been extensively studied, see \cite{Cil} for an account, but very little is known in general. The better understood  cases are linear systems with only double points and linear systems of plane curves.

If all $m_i=2$ there is the famous Alexander--Hirschowitz theorem.
\begin{thm}\cite{AH} \label{thm:AH}$\LL_{n,d}(2^h)$ is special if and only if $(n,d,h)$ is one of the following:
  \begin{itemize}
  \item[i)] $(n,2,h)$ with $2\leq h\leq n$,
\item[ii)] $(2,4,5)$, $(3,4,9)$, $(4,3,7)$, $(4,4,14)$.
  \end{itemize}
\end{thm}
\begin{rmk}\label{rmk:AH-1}
  Further note that for all special linear systems in Theorem~\ref{thm:AH}  ii) the virtual dimension is negative while the dimension is $0$.
\end{rmk}

For linear systems of plane curves there is a very precise conjecture
about speciality. Since we are going to use some of its known cases in
section~\ref{sec:plane} we collect here the necessary notation and statement.

Let $\{p_1,\ldots,p_r\}\subset\p^2$ be general points,
$\nu:\tilde{\p}^2\to\p^2$ their blow up, and
$\LL:=\LL_{2,d}(m_1,\ldots,m_r)(p_1,\ldots,p_r)$. We put a tilde
to indicate the strict transform of curves on $\tilde{\p}^2$. 
\begin{dfn} A $(-1)$-curve $C\subset\p^2$ is a curve such that
  $\tilde{C}$ is a smooth rational curve with self intersection
  $-1$.

  \label{dfn:(-1)_special} The linear system $\LL$ on $\p^2$ is $(-1)$-reducible if 
$$\LL=\sum_{i=1}^k N_iC_i+{\M},$$
where $C =\sum_{i=1}^k C_i$ is a configuration of $(-1)$ curves, $\tilde{\M}\cdot \tilde{C}_i= 0$, for all $i = 1,\ldots,k$, and
$\virtdim({\M})\geq 0$.

The system $\LL$ is called $(-1)$-special if, in addition, there is an $i\in\{1,\ldots,k\}$
such that $N_i > 1$.
\end{dfn}
The leading conjecture for linear systems of plane curves has been formulated by Harbourne and Hirschowitz in 1989.
\begin{con}[Harbourne--Hirschowitz] A linear system of plane curves
  is special if and only if it
is $(-1)$-special.
  \label{con:HH}
\end{con}
The conjecture is known to be true in some cases and in particular Ciliberto--Miranda, \cite{CM}, proved it for 
$\LL_{2,d}(n,m^h)$ for $m\leq 3$.

The sectional genus of a linear system is the geometric genus of a general curve section. We will prove that the linear systems we are interested in have positive sectional genus. For this purpose the following remark is extremely useful.
 \begin{remark}\label{rem:positive_sectional_genus} The genus of any curve of an algebraic system of algebraic curves is not greater
 than the genus of the generic curve of the system, \cite{Ch}.   Therefore to prove that a linear system $\LL$ has positive sectional genus it is enough to show a curve of positive genus in some algebraic family of
 curves whose general member is a curve section of $\LL$.
 \end{remark}

We now reduce Theorem~\ref{th:1} to
Theorem~\ref{thm:Cremona-identifiability} following \cite{Me1}.  
Let $n,d$ be integers. Then a general polynomial
$F\in\C[x_0,\ldots,x_n]_d$ admits a unique decomposition 
$$F=L_1^d+\ldots+L^d_s, $$
with $L_i\in\C[x_0,\ldots,x_n]_1$, if and only if the $s$-secant map
$\pi_s:\sec_s(V_{d,n})\to\p^N$ of the Veronese variety is dominant and
birational, where $N={n+d\choose n}-1$ and $\sec_s(V_{d,n})$ is the abstract $s$-secant variety.
Since $\dim\sec_s(V_{d,n})=s(n+1)-1$, for $\pi_s$ to be birational
there is a numerical constrain. That is 
\begin{equation}
  \label{eq:definition_k(n,d)}
  k(n,d):=\frac{{d+n\choose n}}{n+1}
\end{equation}
has to be an integer. Now let $\Sec_s(V_{n,d})$ be the embedded $s$-secant variety.
For a general point $z\in\Sec_{k(n,d)-1}(V_{n,d})$, let $\varphi:\p^N\dasharrow\p^n$ be the projection from the embedded
tangent space $\mathbb{T}_z\Sec_{k(n,d)-1}(V_{n,d})$.
In \cite[Theorem 2.1]{Me1} it is proved that if $\pi_{k(n,d)}$ is
birational then $\varphi_{|V_{n,d}}$ is birational.
It is well known, see for instance \cite{Me1}, that, 
by Terracini Lemma, this map is associated to the linear system 
$\LL_{n,d}(2^{k(n,d)-1})$.
Therefore the morphism $\pi_{k(n,d)}$ is birational only if  the map associated to $\LL_{n,d}(2^{k(n,d)-1})$ is
birational. Using this, Theorem~\ref{th:1} is a consequence of Theorem~\ref{thm:Cremona-identifiability}.

\section{Limit of double points}\label{sec:limits}

A standard approach to study  the speciality of linear systems $\LL_{n,d}(m_1,\dots,m_r)$ is via degeneration. This is accomplished by using a flat family in which the involved points specialize in some special configuration. Often some of the points are sent  on a hyperplane to apply induction arguments. In our construction we need several degenerations and, unlike the usual set up that concerns degeneration of points in general positions to special position, we need to further degenerate also special positions. 
Thanks to ii) in Proposition~\ref{basic} we may study specializations in a local setup.
We set some notation that will be used throughout the paper.


\begin{dfn} A degeneration is a morphism $\pi: V\to\Delta$, where $\Delta\ni 0,1$ is a complex disk, $V$ is a smooth variety and $\pi$ is proper and flat. For any $t\in\Delta$ we denote the fiber of $\pi$ over $t$ by $V_t$. Let $\sigma_i:\Delta\to V$ be sections of $\pi$  and  $Z$ a scheme with $Z_{\rm red}=\cup_i\sigma_i(\Delta)$.  We let  $Z_t:=Z\cap V_t$, for $t\neq 0$, and $Z_0$ their flat limit.
We say that $Z_0$ is a specialization of $Z_t$.
\end{dfn}

 In this paper we are mainly interested in the following case of specialization.

\begin{construction}[Specialization without collisions]\label{cons:notationlinearsystem}
Let $X$ be the blow up of $\p^n$ in the point $p_1$, with exceptional
divisor $E$, $V:=X\times\Delta$, and $\pi:V\to\Delta$ the canonical projection. Fix $j$ disjoint sections $\{\tau_1,\ldots\tau_j\}$ such that $\tau_i(\Delta)\subset E\times\Delta$ and $r-1$ disjoint sections $\{\sigma_2,\ldots,\sigma_{r}\}$ such that $\sigma_i(\Delta)\cap (E\times\Delta)= \emptyset$. 
Let
\[Z:=\bigcup_{i=2}^r \sigma_i(\Delta)^{m_i}\cup\bigcup_{h=1}^j\tau_h(\Delta)\]
be the scheme supported on the sections with multiplicity $m_i$ along $\sigma_i(\Delta)$.
Let 
$$\LD_{n,d}(m_1[j],m_2\dots,m_r)(p_1[\{\tau_1,\ldots,\tau_j\}],\sigma_2,\ldots,\sigma_r)$$
be the linear subsystem on $V$ associated to divisors of degree $d$ having  multiplicities $m_i$
along $\sigma_i(\Delta)$, $m_1$ in $p_1$ and whose tangent cone contains $\tau_j(\Delta)$. 
Then for any $t\in \Delta$ the linear system
$$\LD_{n,d}(m_1[j],m_2,\dots,m_r)(p_1[\{\tau_1,\ldots,\tau_j\}],\sigma_2,\ldots,\sigma_r)_{|V_t}$$
is 
$$\LL_t:=\LL_{n,d}(m_1[j],m_2,\dots,m_r)(p_1[\{\tau_1(t),\ldots,\tau_j(t)\}],\sigma_2(t),\ldots,\sigma_r(t)).$$
\label{rem:specialize_to_get_non_speciality}
By semicontinuity we have
$$\h^0(V_0,\LL_0)\geq \h^0(V_t,\LL_t).$$
Therefore to prove the nonspeciality of
$\LL_t$ it is 
enough to produce a specialization having
$\LL_0$ nonspecial.
\end{construction}
\begin{dfn}
  \label{def:special_over0} Let $Z_0$  be a specialization of $Z_1$ as in Construction~\ref{cons:notationlinearsystem}. Then we say that the linear system $\LL_0$ is a specialization of  $\LL_1$.
\end{dfn}

\begin{remark}
\label{rem:castelnuovo}  
 Let $H\subset\p^n$ be a hyperplane through the point $p_1$ and $Z_1:=\{p_2^{m_2},\ldots,p_s^{m_s}\}\cup\{t_1,\ldots,t_j\}$ a 0-dimensional scheme as in Construction~\ref{cons:notationlinearsystem}. Assume that $\{p_2,\ldots,p_h\}\subset H$ and $\{t_1,\ldots,t_l\}\subset\T_{p_1}H$, then a classical way to study special linear system is via the {\it Castelnuovo exact sequence} that in this case reads
 \begin{eqnarray*}
   0\to\LL_{n,d-1}((m_1-1)[j-l],m_2-1,\ldots,m_h-1,m_{h+1},\ldots,m_s)\to\\
\to\LL_{n,d}(m_1[j],m_2,\ldots,m_s)\to\LL_{n-1,d}(m_1[l],m_2,\ldots,m_h).
 \end{eqnarray*}

Therefore the nonspeciality of the linear systems 
$$\LL_{n,d-1}((m_1-1)[j-l],m_2-1,\ldots,m_h-1,m_{h+1},\ldots,m_s)\ {\rm and}\ \LL_{n-1,d}(m_1[l],m_2,\ldots,m_h)$$ implies the nonspeciality of $\LL_{n,d}(m_1[j],m_2,\ldots,m_s)$. 
\end{remark}
We state for future reference the following well known fact.
	\begin{lem}\label{indpdti}Let $\LL$ be a linear system on a smooth projective variety $X$ and 
                $C\subset X$ a positive dimensional subvariety. Assume $cod_{\LL}
                |\LL\otimes\I_C|=s$. Set
                $x_1,\dots,x_s\in C$ general points, then
                $x_1,\dots,x_s$ impose independent conditions
                to $\LL$.
	\end{lem}

We apply the above remarks to  prove the nonspeciality of some linear systems we will use along the proof of Theorem~\ref{th:1}.
\begin{pro}
	\label{pro:3piudoppinonspeciale}
	Let $n\geq 3$ and $d\geq 4$, define
	$$r(n,d):=\begin{cases}
	\left\lceil\frac {{n+d\choose n}}{n+1}\right\rceil-n-1 &
	\mbox{ if either } n\neq 3 \mbox{ or } (n,d)=(3,4);\\
	\\
	\left\lceil\frac {{d+3\choose 3}}{4}\right\rceil-5 & \mbox{ if } n=3.\\	
	\end{cases}$$
	The linear system $\LL_{n,d}(3,2^a)$ is nonspecial if $a\leq r(n,d)$.
\end{pro}
\begin{proof}
	We prove the statement by induction on $d$. It is clear that it is enough to prove it for $a=r(n,d)$. 
	The first step of induction is $d=4$ 
	and it is the content of
	\cite[Lemma 2.4]{Pos}. 
	
	Assume $d\geq 5$. 
	Let $Z_1:=\{q^3,p_1^2,\ldots,p_a^2\}$ be a 0-dimensional scheme. Fix a hyperplane $H$ not containing $q$. Let $Z_0$ be a specialization  without collisions of $Z_1$ with
	$$h:=\left\lceil\frac{{d+n-1\choose n-1}}n\right\rceil-1$$
	points on the hyperplane $H$. 
	Since $q\not\in H$ the  Castelnuovo exact sequence reads
	$$0\to \LL_{n,d-1}(3,2^{a-h},1^h)\to \LL_{n,d}(3,2^a)\to\LL_{n-1,d}(2^h),$$
	where the simple base points are all on the hyperplane $H$.
	Since $d\geq 5$ the linear system on the right is nonspecial by Theorem \ref{thm:AH}. Therefore to conclude it is enough to prove that 
	$\LL_{n,d-1}(3,2^{a-h},1^h)$ is nonspecial. 
	
	\begin{claim}
		$\expdim\LL_{n,d-1}(3,2^{a-h},1^h)$ is non negative.
	\end{claim}
	\begin{proof}
		Assume first that $n> 3$. Then 
		$$\expdim\LL_{n,d-1}(3,2^{a-h},1^h)=\binom{n+d-1}{n}-\binom{n+2}{2}-(n+1)a+nh-1\geq
		n^2-\binom{n+2}{2}-1\geq 0.\\
		$$
		Assume that $n=3$. Then
		$$\expdim\LL_{3,d-1}(3,2^{a-h},1^h)=\binom{d+2}{3}-10-4a+3h-1\geq16-14>0.$$
		
	\end{proof}
	We start proving that the simple base points impose independent conditions.
	The points are general in $H$, therefore, by Lemma~\ref{indpdti}, we have to check that  
	$$\dim|\LL_{n,d-2}(3,2^{a-h})|\le 0.$$
This is clear for $d=4$. For $d\geq 5$ by Theorem \ref{thm:AH},
checking also the special cases, we have
	\begin{equation}\label{eq:simplepointssinistra}
		\dim\LL_{n,d-2}(3,2^{a-h})\le\dim\LL_{n,d-2}(2^{a-h+1})=\binom{n+d-2}{n}-1-(n+1)(a-h+1).
	\end{equation}
	For $n=3$ this reads
	\begin{align}
		\dim\LL_{n,d-2}(3,2^{a-h})&\le\binom{d+1}{3}-4\left(\left\lceil\frac {{d+3\choose 3}}{4}\right\rceil-5-\left\lceil\frac {{d+2\choose 2}}{3}\right\rceil+2\right)-1\nonumber\\
		&\le\binom{d+1}{3}-{d+3\choose 3}+\frac{4}{3}{d+2\choose 2}+15\nonumber\\
		&=-\binom{d+1}{2}+\frac{1}{3}{d+2\choose 2}+15\nonumber=\frac{1-d^2}{3}+15\nonumber<0
	\end{align}
	for every $d\ge 7$. We check the rest of the cases by a direct computation.
	\begin{enumerate}
		\item[$(n,d)=(3,5)$] In this case $a=9$ and $h=6$ and the
		linear system
		$\LL_{3,3}(3,2^{3})$ has a unique element. 
		\item[$(n,d)=(3,6)$] In this case $a=16$ and
		$h=9$. An element of the linear system $\LL_{3,4}(3,2^7)$ has to
		contain all quadrics passing through the 8
		singular points. This shows that $\LL_{3,4}(3,2^7)$ is empty.
	\end{enumerate}
	Inequality (\ref{eq:simplepointssinistra}) for $n>3$ yields
	\begin{align}
		\dim\LL_{n,d-2}(3,2^{a-h})&\le\binom{n+d-2}{n}-(n+1)\left(\left\lceil\frac {{n+d\choose n}}{n+1}\right\rceil-\left\lceil\frac {{n+d-1\choose n-1}}{n}\right\rceil-n+1\right)-1\nonumber\\
		&\le\binom{n+d-2}{n}-{n+d\choose n}+{n+d-1\choose n-1}+\frac {{n+d-1\choose n-1}}{n}+n^2+n-1\nonumber\\
		&=-{n+d-2\choose n-1}+\frac {{n+d-1\choose n-1}}{n}+n^2+n-1\nonumber\\
		&=-\frac{(n+d-2)!}{(n-1)!(d-1)!}+\frac {(n+d-1)!}{n!d!}+n^2+n-1\nonumber\\
		&=\frac{(n+d-2)!}{(n-1)!(d-1)!}\left[-1+\frac{n+d-1}{nd}\right]+n^2+n-1\nonumber\\
		&=\binom{n+d-2}{n-1}\cdot\frac{(1-n)(d-1)}{nd}+n^2+n-1.\nonumber
	\end{align}
	The latter decreases as $d$ increases, so 
	\begin{align}
		\dim\LL_{n,d-2}(3,2^{a-h})&\le\binom{n+d-2}{n-1}\cdot\frac{(1-n)(d-1)}{nd}+n^2+n-1\nonumber\\
		&\le\frac{4}{5}\binom{n+3}{4}\cdot\frac{1-n}{n}+n^2+n-1\nonumber\\
		&=\frac{-n^4 - 5n^3 + 25n^2 + 35n - 24}{30}<0\nonumber\end{align}
	for every $n>3$.

	To conclude we prove the nonspeciality of $\LL_{n,d-1}(3,2^{a-h})$.
	For this observe that
\begin{align}
	a-h-&r(n,d-1)\le\frac {{n+d\choose n}}{n+1}-\frac{{d+n-1\choose
            n-1}}{n}-\frac {{n+d-1\choose n}}{n+1}+2\nonumber\\
	&=\frac{1}{n(n+1)}\left[n{n+d\choose n}-(n+1){d+n-1\choose
            n-1} -n{n+d-1\choose n}\right]+2\nonumber\\
&=-\frac{{n+d-1\choose n-1}}{n(n+1)}+2<1\nonumber.
\end{align}
Therefore  we conclude by induction hypothesis.
\end{proof}

\begin{lem}
	\label{lem:simple_points_hyp} Let $\frac{{d+3\choose 3}}4-1>b\geq 1$ and
	$d\geq 4$ be integers. Fix  
	$q\in \Pi\subset\p^n$ a linear space of dimension 3 and $b$ general points on it, say
	$x_1,\ldots, x_b$. Then the linear system
	$$\LL_{n,d}(2^{a},1^b)(q,p_1,\ldots,p_{a-1},x_1,\ldots,x_b)$$
	is nonspecial if 
	$$a\leq a(n,d):= \left\lfloor\frac{{n+d\choose n}-b-1}{n+1}\right\rfloor-\max\{0,n-4\},$$
	where the $p_i$ are general points.
	For the special case $(d,b)=(4,5)$ we prove a better estimate
	$$a(n,4)=\left\lfloor\frac{{n+4\choose
			n}-5-1}{n+1}\right\rfloor-\max\{(n-7),1\}$$
\end{lem}
\begin{proof} Note that since $b\geq 1$, $d\geq 4$ and the virtual dimension is
	non negative then $\LL_{n,d}(2^{a})$ is nonspecial by
	Theorem~\ref{thm:AH}. Therefore we have only to care about the
	simple points.
	
	For $n=3$ the statement is immediate. 
	For $n=4$, by Lemma~\ref{indpdti}, it is enough to check that
	$\dim\LL_{4,d-1}(2^{a(4,d)-1},1)\leq 0$. By
        Theorem~\ref{thm:AH} it is enough to check that 
$\virtdim\LL_{4,d-1}(2^{a(4,d)-1},1)<0$. The latter is a simple computation.
	
	Next we prove the statement by induction on  $n$.
	For $n=i+1\geq 5$ fix a general hyperplane $H\supset\Pi$ and consider
	a degeneration with $a(i,d)$ points on $H$.
	Then, by Castelnuovo exact
	sequence, we
	have to prove the nonspeciality of
	$$\LL_{i,d}(2^{a(i,d)},1^b),\ {\rm and}\
	\LL_{i+1,d-1}(2^{a(i+1,d)-a(i,d)},1^{a(i,d)})$$
	the former is nonspecial by the induction step. For the latter note
	that
	
	\begin{align}\
		a(i+1,d)-a(i,d)&\ge\frac{{i+1+d\choose i+1}-b-1}{i+2}-\frac{{i+d\choose i}-b-1}{i+1}-2\nonumber\\
		&>\frac{{i+1+d\choose i+1}-b-1}{i+2}-\frac{{i+d\choose i}-b-1}{i+2}-2\nonumber\\
		&>\frac{{i+d\choose i+1}}{i+2}-2\ge\frac{{i+d-1\choose i+1}}{i+2}.\nonumber
	\end{align}
	
	For $d\geq 5$ by Theorem~\ref{thm:AH} the linear system
	$\LL_{i+1,d-2}(2^{a(i+1,d)-a(i,d)})$ is empty.
	For $d=4$ it is easy to see that
	$$a(i+1,4)-a(i,4)>i, $$
	and again $\LL_{i,2}(2^{a(i+1,4)-a(i,4)})$ is empty. 
	Let $a(i,d)=\left\lfloor\frac{{i+d\choose
			i}-b-1}{i+1}\right\rfloor-\alpha(i)$, then
	$\alpha(i+1)=\alpha(i)+1$. Then we have
	\begin{equation}
		\label{eq:a-a}
		\virtdim \LL_{i+1,d-1}(2^{a(i+1,d)-a(i,d)},1^{a(i,d)})\geq (i+2)\alpha(i+1)-(i+1)(\alpha(i)+1))-1>0
	\end{equation}
	and
	Lemma~\ref{indpdti} proves that the linear system $
	\LL_{i+1,d-1}(2^{a(i+1,d)-a(i,d)},1^{a(i,d)})$ is non special to 
	conclude this case.
	
	Assume that $d=4$ and $b=5$. We first prove that $\LL_{5,4}(2^{19},1^5)$ is non
	special. Observe that degenerating $12$ double
	points on a hyperplane by Castelnuovo exact sequence we have that
	$\LL_{5,4}(2^{19})$ decomposes in 
	$\LL_{4,4}(2^{12})$ and $\LL_{5,3}(2^7,1^{12})$. It is easy to check
	that these two linear systems are non special and $\dim\LL_{5,3}(2^7,1^{12})=1$.
	The linear system $\LL_{4,3}(2^{11})$ is empty therefore there is at
	most a pencil of divisors in $\LL_{5,4}(2^{19})$ that contains a given
	$\p^3$ through a double point. 
	By hypothesis $\virtdim
	\LL_{5,4}(2^{19},1^5)>1$ hence  this linear
	system  is not
	special. To conclude the statement for $d=4$ we then argue exactly as
	in the first part of the proof, checking Equation~(\ref{eq:a-a}) case
	by case for $n\leq 8$ and then conclude as in the general case.
\end{proof}

At the first stage of our proof we need to allow some points to collapse.

\begin{construction}[Specialization with $h$ collapsing points]\label{not:collapse}
Let $V=\A^n\times\Delta$ and  $\pi:V\to\Delta$ be the canonical projection. Fix $h$ general sections  $\{\sigma_1,\ldots,\sigma_{h}\}$ such that $\sigma_i(0)=q$.
Let $Z:=\cup_i\sigma_i(\Delta)^{m_i}$ and $\nu:V\to \A^n$ be the projection.

Let $X\to V$ be the blow up of $V$ at the point $q$, with exceptional divisor $W$. Then we have natural morphisms
$\nu_X:X\to\A^n$, a degeneration $\pi_{X}:X\to\Delta$, and sections
$\sigma_{X,i}:\Delta\to X$. 
The fiber $X_0$ is given by $W\cup V_0$, 
where $V_0$ is $\A^n$ blown up in one point and $W\cong\p^n$. Let $R=W\cap V_0\cong\p^{n-1}$ be the exceptional divisor of this blow up. 
We want to stress that since the sections $\sigma_i$'s are general $\{\sigma_{X,i}(0)\}$ are general points of $W$. 

\end{construction}

Our main concern is to determine $Z_0$ the flat limit of $Z_t:=Z_{|V_t}$ in Construction~\ref{not:collapse}.
This is in general quite hard and we only have
partial answers to this question. Nonetheless once we understand the limit, we may study the speciality of a linear system via its specializations with collapsing points, using the same technique described in Construction~\ref{rem:specialize_to_get_non_speciality} and Remark~\ref{rem:castelnuovo}.

	\begin{lem}\label{lem:molteplicita_minima} The multiplicity in $q$ of $Z_0$ is at least the minimum integer $j$ such that the linear system $\LL_{n,j}(m_1,\dots,m_h)$  is not empty.
			\end{lem}
                        \begin{proof}
By definition we have
$$\mult_qZ_0\geq\mult_q Z=:\mu.$$
                          Let $\I_{X,Z}$ be the ideal associated to $Z$ on $X$, then  $\I_{X,Z|W}\sim
\LL_{n,\mu}(m_1,\ldots,m_h)$. The ideal $\I_{Z}$ is globally generated therefore the linear system $\I_{X,Z|W} $ has to be non empty.                         \end{proof}



Due to the lack of a classification of special linear systems, it is in general quite hard to determine the exact value predicted by Lemma~\ref{lem:molteplicita_minima}. 
In this paper we are only interested in collapsing of double points, the more general case of arbitrary multiplicities will be treated elsewhere. Theorem~\ref{thm:AH} describes completely the special linear systems of type $\LL_{n,d}(2^h)$.

The next result proves that for double points, outside the small list of special linear systems in Theorem \ref{thm:AH}, the multiplicity of the limit is the one predicted by Lemma~\ref{lem:molteplicita_minima}.

\begin{cor}\label{gradodoppi} Let $\pi$ be a specialization with $h$ collapsing points. Assume that  $m_1=\dots=m_h=2$, $h\geq n+1$ and $(n,h)\neq (2,5), (3,9), (4,7), (4,14)$. Then the multiplicity of $Z_0$ in $q$ is the minimum integer $j$ such that $${n+j\choose n}-h(n+1)>0.$$
\end{cor}

\begin{proof}  As before let $V_0$ be the central fiber and $q\in V_0$ the collapsing point.  Let $Z$ be the scheme supported on the union of $\sigma_i(\Delta)$'s with multiplicity $2$. Let $Z_0$ be the flat limit of $Z_t$, and $\mu$  the minimum integer such that ${n+\mu\choose n}-h(n+1)>0$. The numerical assumption we made, together with Theorem \ref{thm:AH}, ensure that $\mu$ is the minimum integer such that $\LL_{n,\mu}(2^h)$ is not empty and 
\begin{equation}
  \label{eq:AH_dimension}
  \dim \LL_{n,\mu}(2^h)={n+\mu\choose n}-h(n+1)>0.
\end{equation}
Lemma~\ref{lem:molteplicita_minima} tells us $\mult_qZ_0\geq \mu$. Further note that ${n+\mu\choose n}$ is also the degree of the $(\mu+1)$-tuple point of $\p^n$, and $h(n+1)$ is the degree of the 0-dimensional scheme $Z_1$.  By flatness the latter is preserved therefore equation~(\ref{eq:AH_dimension}) shows that $\mult_q Z_0\leq \mu$.
\end{proof}

In some special cases the multiplicity is enough to compute the limit scheme.
			
	\begin{es} \label{5su4plo}
Let us collapse 7 double points in $\pp^2$, then $(n,h)=(2,7)$. The scheme consisting of those points has length $21$. Since $\h^0(\oo_{\pp^2}(5))=15$, they cannot lie on a quintic curve. On the other hand, they lie on a sextic, so the minimum degree we are looking for is $\mu=6$. By Corollary~\ref{gradodoppi} the limit scheme contains a 6-ple point. A 6-ple point in $\pp^2$ has length $21$. Then by ii) in Proposition \ref{stessogrado} the limit scheme is a 6-ple point.

Let $(n,h)=(3,5)$. The scheme consisting of those points has length $20$. Since $\h^0(\oo_{\pp^3}(3))=20$, they cannot lie on a cubic surface. By Corollary~\ref{gradodoppi} the multiplicity we are looking for is $4$. As before observe that a 4-ple point in $\pp^3$ has length $20$. Then by ii) in Proposition \ref{stessogrado} the limit scheme is a 4-ple point.
	\end{es}

When the scheme we are specializing does not have the degree of a multiple point this analysis is not enough to determine the limit scheme.

Let us start with a limit scheme $Z_0\subset \A^n$ and assume we are not in the exceptions of 
 Corollary~\ref{gradodoppi}. Therefore the multiplicity $\mult_qZ_0=\mu$
 is the minimum number such that 
$\LL_{n,\mu}(2^h)$ is not empty. In the setup of Construction~\ref{not:collapse} we have a degeneration
$\pi_{X}:X\to\Delta$, with central fiber $X_0=W\cup V_0$,  $W\cong\p^n$ is the exceptional divisor of the blow up of $q$, and $R=W\cap V_0$. 

Let $\Sigma_X$ be the smooth scheme associated to the sections $\sigma_{X,i}(\Delta)$.
Let $\X\to X$ be the blow up of the ideal sheaf
$\I_{\Sigma_X}$, with exceptional divisors $\E_i$, and $\varphi:\X\to\Delta$ the degeneration  onto $\Delta$.
Note that this blow up is  an isomorphism in a neighborhood of  $V_0$,
then the central fiber is
$$\X_0:=\varphi^{-1}(0)=P\cup V_0,$$
 where $P$ is the blow up of $W$ in $h$ general
points. As before let $R=P\cap V_0$.

Then the linear systems we are interested in are $\LL:=\oo_\X(-\sum_i2\E_i-\mu P)$ and its restrictions $\LL_P$, $\LL_R$, to $P$ and $R$.
The linear system $\LL$ is complete and we aim to understand
when its restrictions stay complete.

Let us start with an instructive example.				
\begin{es}\label{3su3plo}
	Consider $(n,h)=(2,3)$. Then the limit is a  point of multiplicity 3. On
        the other hand $\deg Z_1=9$ while a 3-ple point has degree $6$,
        therefore the limit is not the 3-ple point. There are
        3 more linear conditions the liner system has to satisfy. To
        understand them observe that a plane cubic with 3 double
        points in general position is a union of 3 lines. These lines
        intersect the divisor $R$ providing the further linear
        condition the limit has to satisfy. In particular the linear
        system $\LL_R$ is
 not complete.				\end{es}

	\begin{lem}\label{lem:L_Pcomplete} The linear system $\LL_P$ is complete.
		\begin{proof}
			Consider the exact sequence
			$$0\to\LL(-P)\to\LL\to\LL_P\to 0.$$
		To prove the claim it is enough to show that $\h^1(\LL(-P)) = 0$.
						
			The sheaf $\LL(-P)\sim \oo(\sum_i 2\E_i-(\mu+1)P)$ is the pull-back of the ideal sheaf $\cup_i\I_{\sigma_i(\Delta)}^{2}\cup{\bf m}_{p}^{\mu+1}$ on $\A^n\times\Delta$. Hence we have
$$H^1(\LL(-P)) =H^1(\cup_i\I_{\sigma_i(\Delta)}^{2}\cup{\bf m}_{p}^{\mu+1}) = 0.$$
					\end{proof}
				\end{lem}
                               
The last technical result we need is the following.
				
	\begin{lem} \label{indip} Let $A=\{a_1,\dots,a_{l}\}$ be a set
          of $l$ general points in $\pp^n$, $n\geq 2$, and $R$ a hyperplane such that $A\cap R=\varnothing$. Let
	$$B=\{\langle a_i,a_j\rangle\cap R\}_{i,j\in\{1,\ldots,l\}}.$$
Assume that $l= n+1$, then  $\LL_{n-1,2}(B)$
is nonspecial, that is the points of $B$ impose independent
conditions to quadrics.
		\begin{proof}
 $B$ is a set
of ${n+1\choose2}$ points. Therefore it is enough to prove that there
are no quadrics containing $B$. 
We prove the claim by induction on $n$. For $n=2$ it is easy. 
Let $\Pi_{i}=\langle a_1,\ldots,\hat{a_i},\ldots, a_{n+1}\rangle\cap R$, then by induction
there are no quadrics in $\Pi_{i}$ containing $\Pi_{i}\cap B$. Therefore any
quadric containing $B$ has to contain the hyperplanes $\Pi_{i}$ for any
$i$. This is enough to conclude since $n+1\geq 3$.
\end{proof}
	\end{lem}

\begin{rmk}\label{nongenerici}
Note that even if $B$ imposes independent conditions to quadrics, the
points of $B$ are not in linear general  position. Indeed there are
${n+1\choose t}$ linear spaces of dimension $t-2$ each containing
${t\choose 2}$ points of $B$. For every choice of $t$ points of $A$,
their span is a $\pp^{t-1}$, so the corresponding $\binom{t}{2}$ points of $B$ lie on a $\pp^{t-2}$.
\end{rmk}

We are ready to compute the limits we need.	
					
	\begin{pro} \label{collision1} Let $n\ge 2$ and $Z_1$ a scheme of $n+1$ double points. Let $\pi$
          be a specialization with $n+1$ collapsing points in $\A^n$. Let $Z_0$ be the flat limit, then $Z_0$ is a point of multiplicity 3 together with $\binom{n+1}{2}$ tangent directions.
		\begin{proof}
There are no quadrics singular along $n+1$ general points of $\p^n$ and for $n\geq 2$ the linear system $\LL_{n,3}(2^{n+1})$ has non negative dimension, then  by Corollary \ref{gradodoppi}, the multiplicity of the limit scheme is $3$. Note that the degree of a 3-ple point is $\binom{n+2}{2}$.
						
The base locus of cubics in $\pp^n$ with
                        $(n+1)$ general double points consists of the
                        $\binom{n+1}{2}$ lines joining the
                        points. Each of them cuts a simple point on
                        $R\simeq\pp^{n-1}$. By  Lemma~\ref{indip} these points impose independent conditions to cubics. Then by i) in Proposition~\ref{basic}
$$(n+1)^2=\deg Z_0\geq \binom{n+2}{2}+\binom{n+1}{2}=(n+1)^2.$$ 
Hence the triple point, together with ${n+1\choose 2}$ tangent directions, is a 0-dimensional scheme contained in $Z_0$ and with
the same degree, and  we conclude by Proposition \ref{basic}.\end{proof}
				\end{pro}

				



\section{The induction step}\label{sec:induction}

In this section we develop the induction argument we need to prove Theorem~\ref{thm:Cremona-identifiability}. Thanks to \cite{Me1} we may assume $n\geq d$. We also assume $d\geq 4$. For cubics a different strategy is needed since linear systems of cubics with a triple point and double points are always special. 
s
Fix a point $q\in\p^n$ and a general linear space $\Pi\ni q$, of
dimension $3$. Let $Z_n$ be a scheme having:
\begin{itemize}
\item[-] multiplicity $3$
in $q$ together with ${n+1\choose 2}-s(3,d,n)$ general tangent
directions in $q$ and $s(3,d,n)$ tangent directions on $\Pi$
\item[-] $k(n,d)-n-2-h(3,d,n)$ general double points, $h(3,d,n)$ general double points on $\Pi$.
\end{itemize}
We will define the integers $h(3,d,n)$ and $s(3,d,n)$ later on.

At the linear system level let 
$$\LL_n(d):=|\oo_{\p^n}(d)\otimes\I_{Z_n}|.$$
We aim to prove the nonspeciality of $\LL_n(d)$. For this we
degenerate the scheme $Z_n$ as follows.
Fix a general hyperplane $H\subset\p^n$ containing $\Pi$. Let $Z^{H}_n$ be a
specialization of $Z_n$ such that $H$ contains $h(n-1,d,n)$ double points and $s(n-1,d,n)$ tangent directions, at the point $q$. Let $\LL^H_n(d)$ be the specialized linear system.

By Remark~\ref{rem:specialize_to_get_non_speciality} the nonspeciality of $\LL_n(d)$ is implied by the nonspeciality of
$\LL_{n}^H(d)$. For the latter we use the Castelnuovo exact sequence in Remark~\ref{rem:castelnuovo}. Therefore we are left to prove that the restricted linear system 
$$\LL_{n-1,d}(3[s(n-1,d,n)],2^{h(n-1,d,n)})$$
and the kernel of the restriction map
$$\LL_{n,d-1}(2[s(n,d,n)-s(n-1,d,n)],2^{k(n,d)-n-2-h(n-1,d,n)},1^{h(n-1,d,n)})$$
are nonspecial, with the points satisfying the prescribed requirements. To prove the nonspeciality of these linear systems we set up an induction argument choosing a sequence of integer in such a way that the first linear system has a unique divisor and then use induction for the latter.  
Let us introduce the following notation.
\begin{dfn}
  \label{dfn:induction} Let $s(i,d,n)$ and $h(i,d,n)$ be non-negative integers. Assume $n\geq d$ and $d\geq 4$.
Fix a general flag of linear spaces 
$$H_2\subset H_3\subset H_4\subset\ldots\subset H_{n-1}\subset H_n=\p^n,$$
with $H_i\cong\p^i$, $H_3=\Pi$ and $q\in H_2$. For $i\in\{2,\dots,n\}$, consider a 0-dimensional scheme $Z_i\subset H_i$ such that:
\begin{itemize}
\item[-] $Z_{i-1}$ is a flat limit of $Z_{i|H_{i-1}}$,
\item[-] $Z_i$ has multiplicity $3$
in $q$ together with $s(i,d,n)$ tangent directions,  $s(i-1,d,n)$ of which are tangent directions supported in $H_{i-1}$, 
\item[-] for every $i\in\{3,\dots,n\}$, $Z_i$ has $h(i,d,n)$ double points, $h(i-1,d,n)$ of which are supported in $H_{i-1}$.
\item[-] $Z_2$ has $h(2,d,n)$ double points.
\end{itemize}
Let 
$$\LL^{H}_{i}(d)=|\oo_{\p^i}(d)\otimes\I_{Z_i}|.$$
\end{dfn}

The following is the first step toward proving the induction we need.
\begin{lem}
  \label{lem:firstdegeneration} Assume that, for any $n\geq i\geq 2$
  there are integers $h(i,d,n)\geq h(i-1,d,n)$ and
  $s(i,d,n)\geq s(i-1,d,n)$ such that
  \begin{itemize}
  \item[i)]  $\expdim\LL^H_{i}(d)=i$,
\item[ii)] $\dim\LL^{H}_{i-1}(d)=i-1$,
\item[iii)]   $\expdim\LL^{H}_{i}(d)-H_{i-1}=0$,
\item[iv)] $\dim\LL^{H}_{i}(d)-2H_{i-1}\le 0$,
\item[v)] there is at most one divisor $D\in
  \LL^{H}_{i}(d)$ with  $\mult_qD>3$,
\item[vi)]  for $i> 3$,
  $\LL^{H}_{i-1}(d)_{|\Pi}=\LL_{3,d}(3[s(3,d,n)],2^{h(3,d,n)})$.
\item[vii)]  $h(i-1,d,n)-h(3,d,n)+s(i,d,n)-s(i-1,d,n)>(i-4)(i+1)$, for
  $i\geq 5$ and $d\geq 5$,

 $h(i-1,4,n)-h(3,4,n)+s(i,4,n)-s(i-1,4,n)>(i+1)$, for
  $5\leq i\leq 8$,

 $h(i-1,4,n)-h(3,4,n)+s(i,4,n)-s(i-1,4,n)>(i-7)(i+1)$, for
  $i\geq 9$.
  \end{itemize}
Then
\begin{enumerate}
\item $\dim\LL^{H}_{i}(d)-H_{i-1}=0$,
\item the divisor in $\LL^{H}_{i}(d)-H_{i-1}$ does not contain $\Pi$,
\item $\LL^{H}_{i}(d)$ is nonspecial,
\item for $i\geq 3$, $\LL^{H}_{i}(d)_{|\Pi}=\LL_{3,d}(3[s(3,d,n)],2^{h(3,d,n)})$.
\end{enumerate}

\end{lem}
\begin{proof} By assumption i) $h(3,d,n)<\frac{{d+3\choose 3}}4-1$, the linear system $\LL^{H}_{i}(d)-H_{i-1}$ is 
$$\LL_{i,d-1}(2[s(i,d,n)-s(i-1,d,n)],2^{h(i,d,n)-h(i-1,d,n)}, 1^{h(i-1,d,n)}),$$
and it has non negative expected dimension by iii). For $i=4$, using
Lemma~\ref{indpdti}, it is easy to check that the
simple points on $\Pi$ impose independent conditions  and the general element in  
$\LL^H_4(d)-\Pi$ does not contain $\Pi$.  By vii) and 
Lemma~\ref{lem:simple_points_hyp}, for $i>4$, the simple
base points on $\Pi$ impose independent conditions and the general element in  
$\LL^H_i(d)-H_{i-1}$ does not contain $\Pi$. This shows (2).  

By assumption iv) and Lemma~\ref{indpdti}, the other simple base points on $H_{i-1}$ impose independent conditions to $\LL_{i,d-1}(2^{h(i,d,n)-h(i-1,d,n)+1}).$

 Similarly assumption v) and Lemma~\ref{indpdti} ensure that the tangential directions in $q$ impose independent conditions to $\LL_{i,d-1}(2^{h(i,d,n)-h(i-1,d,n)+1}).$
Therefore we are left to prove that $\LL_{i,d-1}(2^{h(i,d,n)-h(i-1,d,n)+1})$ is
nonspecial. The latter is a linear system with only double points and positive expected dimension. Then by Theorem~\ref{thm:AH} we know it is nonspecial, keep also in mind Remark~\ref{rmk:AH-1}. This proves (1).

Then from the Castelnuovo exact sequence and assumption
ii) the linear system  $\LL^{H}_{i}(d)$ is non
special. To conclude observe that the nonspeciality of
$\LL^{H}_{i}(d)$, (1) and vi) yield (4).
\end{proof}

To apply Lemma~\ref{lem:firstdegeneration}
we have first to produce the sequences of integers $h(i,d,n)$ and $s(i,d,n)$.

\begin{pro}\label{inductionstep} Fix integers $n\geq d\geq 4$, and let $i\in\{3,\dots,n\}$. Assume that the number $k(n,d)$ defined in formula \ref{eq:definition_k(n,d)} is an integer
	. Then there are sequences $\{h(i,d,n)\}_{i\in\{2,\dots,n\}}$ and\\
	$\{s(i,d,n)\}_{i\in\{2,\dots,n\}}$ such that
	\begin{enumerate}
		\item $\expdim \LL_{i,d-1}(2[s(i,d,n)-s(i-1,d,n)],2^{h(i,d,n)-h(i-1,d,n)},1^{h(i-1,d,n)})=0,$
		\item $\expdim \LL_{i-1,d}(3[s(i-1,d,n)],2^{h(i-1,d,n)})=i-1,$
	\end{enumerate}
	and the following properties hold:
	\begin{enumerate}
		\item[i)] $h(n,d,n)=k(n,d)-1-n-1$,  
		$$h(n,d,n)-h(n-1,d,n)\geq \frac{d-1}{(n+2)(n+1)}\binom{n+d}{n} -3>0,$$
		$$ h(i+1,d,n)-h(i,d,n)\ge\frac{d-1}{(i+2)(i+1)}\binom{i+d}{i} -2>0$$ for every $i\in\{2,\dots,n-2\}$ and  $h(3,d,n)<{d+2\choose 3}-4$,
		\item[ii)] $s(n,d,n)=\binom{n+1}{2}$ and
		$$s(i-1,d,n)\in\left\{\frac{i^2-3i-2}{2},\dots,\frac{i^2-i-4}{2}\right\},$$
		\item[iii)]  $s(2,d,n)\ge 0$,
		\item[iv)] $s(i,d,n)\geq s(i-1,d,n)$ for every $i\in\{4,\dots,n-1\}$,
		\item[v)]  $s(i,d,n)-s(i-1,d,n)<{i+1\choose 2}$,
\item[vi)] $h(i-1,d,n)-h(3,d,n)+s(i,d,n)-s(i-1,d,n)>(i-4)(i+1)$, for
  $i\geq 5$ and $d\geq 5$,

 $h(i-1,4,n)-h(3,4,n)+s(i,4,n)-s(i-1,4,n)>(i+1)$, for
  $5\leq i\leq 8$,

 $h(i-1,4,n)-h(3,4,n)+s(i,4,n)-s(i-1,4,n)>(i-7)(i+1)$, for
  $i\geq 9$.
	\end{enumerate}
	\begin{proof}   
		
		To simplify notation we set for the moment
		$s_i:=s(i,d,n)$,  and $h_{i}:=h(i,d,n)$, $k:=k(n,d)$. Set $s_n=\binom{n+1}{2}$ and $h_n=k-1-n-1$. For $i\in\{3,\dots,n\}$ define
		$$a_i:=:a(i,d):=\binom{i+d-1}{i-1}-\frac{3i}{2}-\frac{i^2}{2}.$$
		
		The expected dimension of the linear system $\LL^H_{i,d-1}(2[s_{i}-s_{i-1}],2^{h_i-h_{i-1}},1^{h_{i-1}})$ is
		\[\exp_i:=\binom{d-1+i}{i}-1-(i+1)(h_i-h_{i-1}+1)-\left(s_i-s_{i-1}\right)-h_{i-1} \]
		Then assumption (1) reads $\exp_n=0$ and yields
		\[\binom{d-1+n}{n}-1-(n+1)(k-n-2-h_{n-1}+1)-h_{n-1}-\left(\binom{n+1}{2}-s_{n-1}\right)=0\]
		\begin{align}
		nh_{n-1}+s_{n-1}&=-\binom{d-1+n}{n}+2+(n+1)k-(n+1)(n+2)+n+\binom{n+1}{2}\nonumber\\
		&=-\binom{d-1+n}{n}+2+\binom{n+d}{n}-(n+1)(n+2)+n+\binom{n+1}{2}\nonumber\\
		&=\binom{n+d-1}{n-1}+2-(n+1)(n+2)+n+\binom{n+1}{2}\nonumber=a_n.
		\end{align}
		
Therefore $h_{n-1}$ and $s_{n-1}$ satisfy the following equation
		$$nh_{n-1}+s_{n-1}=a_n.$$
		Note that
		$\frac{n^2-n-4}{2}-\frac{n^2-3n-2}{2}=n-1$, therefore there  is a unique
		$t\in\left\lbrace
		\frac{n^2-3n-2}{2},\dots,\frac{n^2-n-4}{2}\right\rbrace $ such that
		$a_n-t$ is a multiple of $n$; call $s_{n-1}$ that number and
		define \[h_{n-1}=\frac{a_n-s_{n-1}}{n}.\]
		This also settles $(2)$, for $i=n$, since by construction  
		$$\exp\dim \LL_{n,d}(3[s(n,d,n)],2^{h(n,d,n)})=n$$ hence
		$$n=\exp_n+\exp\dim \LL_{n-1,d}(3[s(n-1,d,n)],2^{h(n-1,d,n)})+1.$$
		
		In a similar fashion $\exp_{n-1}=0$ gives
		\[(n-1)h_{n-2}+s_{n-2}=\binom{n+d-2}{n-2}-\frac{n^2}{2}-\frac{n}{2}+1=a_{n-1}.\]
		As before $\frac{n^2-3n-2}{2}-\frac{n^2-5n+2}{2}=n-2$, and there is a
		unique $t\in\left\lbrace
		\frac{n^2-5n+2}{2},\dots,\frac{n^2-3n-2}{2}\right\rbrace $ such that
		$a_{n-1}-t$ is a multiple of $n-1$; call $s_{n-2}$ that number, and define \[h_{n-2}=\frac{a_{n-1}-s_{n-2}}{n-1}.\]
		We iterate the argument, defining
		$s_{i-1}$ to be the only natural number in $\left\lbrace
		\frac{i^2-3i-2}{2},\dots,\frac{i^2-i-4}{2}\right\rbrace $ such that
		$i\mid a_i-s_{i-1}$. Hence we have
		\begin{equation}
		\label{eq:hsa}
		ih_{i-1}+s_{i-1}=a_i 
		\end{equation}
		and conditions ii) and iv) are satisfied
		.
		
		
		Next we check $s(2,d,n)\ge 0$. By defini
tion $s(2,d,n)\ge -1$.
		Assume that $s(2,d,n)=-1$. Then
		\[a(3,d)=\frac{(d+1)(d+2)}{2}-9\equiv -1\mbox{ (mod 3)}\]
		so $(d+1)(d+2)=6t+4\equiv 1$ (mod 3), and this is impossible because 1
		is irreducible in $\Z_3$. Therefore $s(2,d,n)\ge 0$ for any $d$.
		
		Let us check condition v)
		. Assume first that $i<n$. Then 
		$$2(s(i,d,n)-s(i-1,d,n))\leq (i+1)^2-i-5-i^2+3i+2=4i-2<(i+1)i, $$
		for $i\geq 3$.
		Assume $i=n$, then
		$$2(s(n,d,n)-s(n-1,d,n))\leq 2{n+1\choose 2}-n^2+3n+2=4n+2<n(n+1), $$
		for $n\geq 4$.
		
		Next we focus on  condition i).
		\begin{claim}\label{cl:disuguglianza_h}
			Set $i\ge 2
			$, then
			\begin{equation}
			h_n-h_{n-1}\ge \frac{d-1}{(n+2)(n+1)}\binom{n+d}{n}-3>0\nonumber
			\end{equation}
			and
			\begin{equation}
			\label{eq:h(i+1)-h(i)}
			h_{i+1}-h_{i}\ge\frac{d-1}{(i+2)(i+1)}\binom{i+d}{i}-2>0\mbox{ for every } i\le n-2.\nonumber
			\end{equation}
			\begin{proof}[Proof of the Claim]
				First assume $i=n-1$. Then we have, by Equation~(\ref{eq:hsa}),
				\begin{align}
				h_n-h_{n-1} & =\frac{\binom{n+d}{n}}{n+1}-n-3-\frac{a_n-s_{n-1}}{n}\geq\nonumber\\
				& \geq\frac{\binom{n+d}{n}}{n+1}-\frac{\binom{n+d-1}{n-1}}{n}-3= \nonumber\\
				& =\frac{1}{n(n+1)}\left[ n\binom{n+d}{n}-(n+1)\binom{n+d-1}{n-1}\right] -3= \nonumber\\
				& =\frac{d-1}{(n+1)n}\binom{n+d-1}{n-1} -3.\nonumber
				\end{align}
				The case $i<n-1$ is similar but a bit more painful, keep in mind Equation~(\ref{eq:hsa})
				\begin{align}
				h_{i+1}-h_i & =\frac{a_{i+2}-s_{i+1}}{i+2}-\frac{a_{i+1}-s_i}{i+1}\nonumber\\
				& \ge\frac{(i+1)a_{i+2}-(i+1)\cdot\frac{(i+2)^2-i-6}{2}-(i+2)a_{i+1}+(i+2)\cdot\frac{(i+1)^2-3(i+1)-2}{2}}{(i+2)(i+1)}\nonumber\\
				& =\frac{(i+1)a_{i+2}-(i+2)a_{i+1}}{(i+2)(i+1)}-\frac{3i^2+7i+6}{2(i+2)(i+1)}\nonumber\\
				& =\frac{1}{(i+2)(i+1)}\left[ (i+1)a_{i+2}-(i+2)a_{i+1}\right] -\frac{(i+1)(3i+4)}{2(i+2)(i+1)}-1\nonumber\\
				&=\frac{1}{(i+2)(i+1)}\left[ (i+1) \binom{i+d+1}{i+1}-(i+2)\binom{i+d}{i}\right]-\frac{1}{2}-\frac{3i+4}{2(i+2)}-1\nonumber\\
				&=\frac{1}{(i+2)(i+1)}\left[(i+1) \binom{i+d}{i+1}-\binom{i+d}{i}\right] -\frac{2i+3}{i+2}-1\nonumber\\
				& \ge\frac{1}{(i+2)(i+1)}\left[i\frac{(i+d)!}{(i+1)!(d-1)!}+\frac{(i+d)!}{(i+1)!(d-1)!}-\frac{(i+d)!}{i!d!}\right]-\frac{2i+4}{i+2}-1\nonumber\\
				& =\frac{(i+d)!}{i!(i+2)(i+1)(d-1)!}\left[\frac{i}{i+1}+\frac{1}{i+1}-\frac{1}{d}\right]-3\nonumber\\
				& =\frac{(i+d)!(d-1)}{i!d!(i+2)(i+1)}-3\nonumber\\
				&=\frac{d-1}{(i+2)(i+1)}\binom{i+d}{i}-3,\nonumber
				\end{align}
				
				
				

				
				
				for every $i\ge 2$.
				Note that $h_{i+1}-h_i$ increases as $d$ does. For $d=4$ we have
				\begin{align}
				h(i+1,4,n)-h(i,4,n) & \ge\frac{3}{(i+2)(i+1)}\binom{i+4}{4}-3=\frac{i^2 + 7i -12}{8}>0,\nonumber
				\end{align}
				for $i\geq 2$.
			\end{proof}
		\end{claim}
The Claim proves i). We are left with vi). Assume first $i\ge 5$ and $d\ge 5$.
\begin{align}
	h(i-1,d,n)&-h(3,d,n)+s(i,d,n)-s(i-1,d,n)-(i-4)(i+1)\nonumber\\
	&\ge h(i-1,d,n)-h(3,d,n)-(i-4)(i+1)\nonumber\\
	&=\frac{a(i,d)-s(i-1,d,n)}{i}-\frac{a(4,d)-s(3,d,n)}{4}-(i-4)(i+1)\nonumber\\
	&\ge\frac{\binom{i+d-1}{i-1}-\frac{3i}{2}-\frac{i^2}{2}-\frac{i^2-i-4}{2}}{i}-\frac{\binom{d+3}{3}-14-1}{4}-(i-4)(i+1).\nonumber
\end{align}
The latter increases as $d$ does, so
\begin{align}
	&\frac{\binom{i+d-1}{i-1}-\frac{3i}{2}-\frac{i^2}{2}-\frac{i^2-i-4}{2}}{i}-\frac{\binom{d+3}{3}-14-1}{4}-(i-4)(i+1)\nonumber\\
	\ge& \frac{\binom{i+4}{5}-\frac{3i}{2}-\frac{i^2}{2}-\frac{i^2-i-4}{2}}{i}-\frac{\binom{8}{3}-14-1}{4}-(i-4)(i+1)\nonumber\\
	=&\frac{i^5 + 10i^4 - 85i^3 + 290i^2 - 846i + 240}{120i}>0\nonumber
\end{align}
for every $i\ge 5$. Now assume $d=4$. For $5\leq i\le 8$ we have
\begin{align}
	h(i-1,4,n)&-h(3,4,n)+s(i,4,n)-s(i-1,4,n)-i-1\nonumber\\
	&=h(i-1,4,n)+s(i,4,n)-s(i-1,4,n)-i-6\nonumber\\
	&\ge h(i-1,4,n)-i-6\nonumber\\
	&= \frac{a(i,4)-s(i-1,4,n)}{i}-i-6\nonumber\\
	&\ge\frac{\binom{i+3}{4}-\frac{3i}{2}-\frac{i^2}{2}-\frac{i^2-i-4}{2}}{i}-i-6\nonumber\\
	&=\frac{i^4 + 6i^3 - 37i^2 - 162i + 48}{24i}>0\nonumber
\end{align} for $i\ge 6$. For $i=5$ we have
\[h(4,4,n)-h(3,4,n)+s(5,4,n)-s(4,4,n)-6=9-5+9-5-6>0.\]
Assume now $i\ge 9$.
\begin{align}
	h(i-1,4,n)&-h(3,4,n)+s(i,4,n)-s(i-1,4,n)-(i+1)(i-7)\nonumber\\
	&\ge\frac{a(i,4)-s(i-1,4,n)}{i}-5-(i+1)(i-7)\nonumber\\
	&\ge\frac{\binom{i+3}{4}-\frac{3i}{2}-\frac{i^2}{2}-\frac{i^2-i-4}{2}}{i}-5-(i+1)(i-7)\nonumber\\
	&=\frac{i^4 - 18i^3 + 131i^2 + 30i + 48}{24i}>0\nonumber
\end{align} for every $i\ge 9$.\end{proof}
\end{pro}

\begin{remark} There is an interesting consequence of
  Proposition~\ref{inductionstep}. The sequences $h(i,d,n)$ and
  $s(i,d,n)$ do not vary with $n$, as long as $i<n$. This is crucial
  for all the computations we are going to do and opens also
  interesting generalization of our arguments that we will explore in
  the future.
\end{remark}

\begin{es}
	Here we present the computation in a specific case. Assume $n=5$ and $d=4$. By definition $k(5,4)=21\in\N$. Following the proof of Proposition \ref{inductionstep}, we set $s(5,4,5)=\binom{5}{2}=10$ and $h(5,4,5)=21-2-5=14$. Next we compute
	\begin{align}
	a(5,4) &=\binom{8}{4}-\frac{15}{2}-\frac{25}{2}=50,\nonumber\\
	a(4,4) &=\binom{7}{3}-6-8=21\nonumber\mbox{ and}\\
	a(3,4) &=\binom{6}{2}-\frac{9}{2}-\frac{9}{2}=6.\nonumber
	\end{align}
	Now $s(4,5,4)$ is the only natural number $t\in\{4,\dots,8\}$ such that $5\mid 50-t$. This means $s(4,5,4)=5$ and therefore $h(4,5,4)=\frac{50-5}{5}=9$. In the same way $s(3,5,4)$ is the only number $t\in\{1,\dots,4\}$ such that $4\mid 21-t$. Again this implies $s(3,5,4)=1$ and $h(3,5,4)=\frac{21-1}{4}=5$. Finally $s(2,5,4)$ is the only number $t\in\{0,1\}$ such that $3\mid 6-t$, so we conclude that $s(2,5,4)=0$ and $h(2,5,4)=3$.
\end{es}

From now on we fix the sequences of integers $h(i,d,n)$ and $s(i,d,n)$ of Proposition~\ref{inductionstep}. The following proves that Lemma~\ref{lem:firstdegeneration} assumption iv) is satisfied for this choice of integers.

		\begin{lem}\label{accanonbase}
			Assume $n\ge d\ge 4$ and $(n,d)\neq (4,4)$. Then
			$\LL^{H}_{i}(d)-2H_{i-1}$ is
			empty for every $i\in\{3,\dots,n\}$
			.
			\begin{proof} 
				By definition $\LL^{H}_{i}(d)-2H_{i-1}$ is
				$$\LL_{i,d-2}(1[s(i,d,n)-s(i-1,d,n)],2^{h(i,d,n)-h(i-1,d,n)}).$$
				
First assume $d=4$, $n\ge 5$. Consider $i=3$. A direct computation shows
$h(3,4,n)=5$, $h(2,4,4)=2$, therefore $h(3,4,n)-h(2,4,n)\ge 3=i$. Consider $i\in\{4,\dots,n-1\}$.
By i) in Proposition~\ref{inductionstep} we have
\[h(i,4,n)-h(i-1,4,n)\ge\frac{3}{(i+1)i}\binom{i+3}{4}-2=\frac{i^2 + 5i - 10}{8}\ge i.\]
Consider $i=n>d=4$. By i) in Proposition~\ref{inductionstep} we have
\[h(n,4,n)-h(n-1,4,n)\ge\frac{3}{(n+1)n}\binom{n+3}{4}-3=\frac{n^2+5n-18}{8}\geq n\]
for $n\ge 6$. If $n=5$ we compute $h(5,4,5)-h(4,4,5)=14-9=5=i$.
In all these cases the linear system we are interested in is contained in $\LL_{n,2}(1,2^n)$ which is empty. 				
Next we consider exceptions in Theorem~\ref{thm:AH} with $d\geq 4$. In our notation these are the cases $(i,d)=(3,6), (4,6), (4,5)$. A  direct computation shows that  $h(i,d,n)-h(i-1,d,n)$ is greater than the exceptional value of Theorem \ref{thm:AH}.

				
				
Assume then $n\ge d>4$. By hypothesis the linear system $\LL_{i,d-2}(2^{h(i,d,n)-h(i-1,d,n)})$ is nonspecial. Then by i) in Proposition~\ref{inductionstep}, for $i<n$ we have
				\begin{align}
				\virtdim\LL_{i,d-2}( & 1[s(i,d,n)-s(i-1,d,n)],2^{h(i,d,n)-h(i-1,d,n)})\nonumber\\
				&<\binom{d+i-2}{i}-(i+1)(h(i,d,n)-h(i-1,d,n))\nonumber\\	 &\le\binom{d+i-2}{i}-(i+1)\left[\frac{d-1}{(i+1)i}\binom{i+d-1}{i-1} -2\right]\nonumber\\
				& =\frac{(d+i-2)!}{i!(d-2)!}-\frac{d-1}{i}\cdot\frac{(i+d-1)!}{(i-1)!d!} +2(i+1)\nonumber\\
				& =\frac{(i+d-2)!}{i!(d-2)!}-(d-1)\cdot\frac{(i+d-1)!}{i!d!} +2(i+1)\nonumber\\
				& =\frac{(i+d-2)!}{i!(d-2)!}\left[ 1-(d-1)\cdot\frac{i+d-1}{d(d-1)}\right] +2(i+1)\nonumber\\
				& =\binom{i+d-2}{d-2}\cdot\frac{1-i}{d} +2(i+1).\nonumber
\end{align}
				Note that the latter decreases as $d$ increases. For $d=5$ we get
				\begin{align}
				\frac{1-i}{5} \binom{i+3}{3}+2(i+1)\le\frac{-i^4 - 5i^3 - 5i^2 + 65i +66}{30}\leq0\nonumber
				\end{align}
				for every $i\ge 3$. 

To conclude assume $i=n$. Then the virtual dimension is bounded as follows
				\begin{align}
				\virtdim& \LL_{n,d-2}( 1[s(n,d,n)-s(n-1,d,n)],2^{h(n,d,n)-h(n-1,d,n)})
				\nonumber\\	&
				<\frac{1-n}{d}\binom{n+d-2}{d-2} +3(n+1).\nonumber
				\end{align}
				As before it decreases as $d$ increase. For $d=5$ we have
				\begin{align}
			\frac{1-n}{5}\binom{n+3}{3} +3(n+1)&=\frac{-n^4 - 5n^3 - 5n^2 + 95n + 96}{30}<0,\nonumber
				\end{align}
				for every $n\ge 5$.
			\end{proof}
		\end{lem}

We are in the condition to state and prove the induction step we described.
\begin{pro}
  \label{pro:induction_step_tuned}
Assume $n\ge d\ge 4$ and $(n,d)\neq (4,4)$. Let $i\in\{3,\dots,n\}$. Suppose that
\begin{itemize}
\item[a)] $\LL^H_{i-1}(d)$ is nonspecial
\item[b)] there is at most one divisor $D\in
  \LL^H_{i-1}(d)$ with $\mult_q D>3$
\end{itemize}
 then
$\LL^H_{i}(d)$ is nonspecial and there is at
most one divisor $D\in \LL^H_{i}(d)$ with $\mult_q D>3$.
\end{pro}
\begin{proof} Recall $q\in H_{i-1}$. First we  check that the
  conditions i), ii), iii), iv), vii) in
  Lemma~\ref{lem:firstdegeneration} are satisfied. 
Point i) is (2) in Proposition~\ref{inductionstep}, ii) is a) and (2) in Proposition~\ref{inductionstep}, 
iii) is
(1) in Proposition~\ref{inductionstep}, iv) is
Lemma~\ref{accanonbase}, and vii) is vi) in Proposition~\ref{inductionstep}.

We are left to prove that there is at most  one divisor $D\in
\LL^{H}_{i}(d)$ with $\mult_qD>3$.
All divisors $D\in \LL_{i}^H(d)$ with $\mult_qD>3$ either contain
$H_{i-1}$ or restricts to
divisors $D_{|H_{i-1}}\in\LL_{i-1}^H(d)$ with
$\mult_qD_{|H_{i-1}}>3$. On the other hand, by assumption b), if there is a
pencil of these divisors the unique divisor in $\LL_{i}^H(d)-H_{i-1}$
has multiplicity at least $4$ in $q$.
Therefore to conclude it is enough to prove that $\mult_qD=3$ for the divisor $D$ with $D\supset
H_{i-1}$. 
The divisor  
$D-H_{i-1}$ is in
$$\LL_{i,d-1}(2[s(i,d,n)-s(i-1,d,n)],2^{h(i,d,n)-h(i-1,d,n)},1^{h(i-1,d,n)}). $$
A straightforward computation shows that
$$ h(i,d,n)-h(i-1,d,n)<\left\lceil\frac{{i+d-1\choose i}}{i+1}\right\rceil -i-1,$$
therefore by Proposition~\ref{pro:3piudoppinonspeciale} the linear system $\LL_{i,d-1}(3,2^{h(i,d,n)-h(i-1,d,n)})$ is nonspecial.
By Lemma~\ref{accanonbase}
$\LL_{i,d-1}(2[s(i,d,n)-s(i-1,d,n)],2^{h(i,d,n)-h(i-1,d,n)})-H_{i-1}$is
empty and so is
$\LL_{i,d-1}(3,2^{h(i,d,n)-h(i-1,d,n)})-H_{i-1}$.

 Hence, arguing as in
Proposition~\ref{lem:firstdegeneration}, we use 
Lemma~\ref{lem:simple_points_hyp} and Lemma~\ref{indpdti} to ensure that the simple points on the hyperplane impose
independent conditions and the linear system $\LL_{i,d-1}(3,2^{h(i,d,n)-h(i-1,d,n)},1^{h(i-1,d,n)})$ is nonspecial. 
Point (1) in Proposition~\ref{inductionstep} gives
$$\expdim
\LL_{i,d-1}(2[s(i,d,n)-s(i-1,d,n)],2^{h(i,d,n)-h(i-1,d,n)},1^{h(i-1,d,n)})=0.$$
Point v) in  Proposition~\ref{inductionstep} gives
$s(i,d,n)-s(i-1,d,n)<{i+1\choose2}$, then  
\begin{eqnarray*}
  \virtdim\LL_{i,d-1}(3,2^{h(i,d,n)-h(i-1,d,n)},1^{h(i-1,d,n)})<\\
<\virtdim\LL_{i,d-1}(2[s(i,d,n)-s(i-1,d,n)],2^{h(i,d,n)-h(i-1,d,n)},1^{h(i-1,d,n)})=0.\end{eqnarray*}
Since $\LL_{i,d-1}(3,2^{h(i,d,n)-h(i-1,d,n)},1^{h(i-1,d,n)})$ is nonspecial it is empty and any divisor in $\LL_{i,d-1}(2[s(i,d,n)-s(i-1,d,n)],2^{h(i,d,n)-h(i-1,d,n)},1^{h(i-1,d,n)})$ has a double point in $q$.
\end{proof} 

The next proposition proves the first step of our induction argument.

\begin{pro}
  \label{pro:step1_induction} Assume $n\ge d\ge 4$ and $(n,d)\neq (4,4)$. Then the linear system
  $\LL_{2,d}(3[s(2,d,n)],2^{h(2,d,n)})$ is nonspecial and there is at most
  one divisor $D\in \LL_{2,d}(3[s(2,d,n)],2^{h(2,d,n)})$ with $\mult_q D>3$.
 \end{pro}

		\begin{proof} A simple check of the list in \cite[Theorem 7.1]{CM} shows that the linear systems $\LL_{2,d}(3,2^{h(2,d,n)})$ and $\LL_{2,d}(4,2^{h(2,d,n)})$ are nonspecial for $d\geq 5$. While a direct computation shows that 
$\LL_{2,4}(3,2^{h(2,4,n)})$ is nonspecial and $\dim\LL_{2,4}(4,2^{h(2,4,n)})=0$. In particular $\LL_{2,d}(4,2^{h(2,d,n)})$ is empty for $d\geq 5$ and has dimension 0 for $d=4$.

We are left to study the $s(2,d,n)$ tangent direction. 
If $s(2,d,n)=0$, we are done. Suppose $s(2,d,n)=1$. This is possible only
for $d\ge 5$. Since $\LL_{2,d}(4,2^{h(2,d,n)})$ is empty  and
$\LL_{2,d}(3,2^{h(2,d,n)})$ is of
positive dimension we conclude, by Lemma~\ref{indpdti}, that  $\LL_{2,d}(3[s(2,d,n)],2^{h(2,d,n)})$ is nonspecial.
Moreover for $d\geq 5$ any divisor in $\LL_{2,d}(3,2^{h(2,d,n)})$ has
multiplicity $3$ in $q$, while there is a unique divisor with
multiplicity $4$ in $\LL_{2,4}(3,2^{h(2,4,n)})$.
		\end{proof}

We are ready to state the nonspeciality result we were looking for.
\begin{pro}
  \label{pro:induction} If $n\ge d\ge 4$ and $(n,d)\neq (4,4)$, then
  \begin{itemize}
  \item[i)] the linear system
  $\LL^H_{n}(d)$ and $\LL_n(d)$  are nonspecial,
\item[ii)]  there is at most
  one divisor $D\in \LL^H_{n}(d)$ with $\mult_q D>3$.
  \end{itemize}
\end{pro}
\begin{proof} By Proposition~\ref{pro:step1_induction}
  $\LL_{2,d}(3[s(2,d,n)],2^{h(2,d,n)})$ satisfies i) and ii). Then by
  Proposition~\ref{pro:induction_step_tuned} $\LL_n^H(d)$ satisfies i)
  and ii). We already observed that $\LL_n(d)$ is nonspecial if
  $\LL_n^H(d)$ is nonspecial.
\end{proof}

\section{The genus bound}\label{sec:plane}

In this section we aim to bound from below the sectional genus of linear systems $\LL_n^H(d)$. We start from $\LL_{2,d}(3[s(2,d,n)],2^{h(2,d,n)})$.
\begin{pro}
  \label{pro:degp2}
The sectional genus of $\LL_{2,d}(3[s(2,d,n)],2^{h(2,d,n)})$ is
$0$ for $d=4,5$ and it is positive for $d\geq 6$. 
\end{pro}
\begin{proof}
 The general element in  $\LL_{2,d}(3[s(2,d,n)],2^{h(2,d,n)})$ has a
 triple point in $q$, by Proposition~\ref{pro:step1_induction}, and
 double points at the rest of the assigned points by \cite[Theorem 8.1]{CM}. 
\begin{claim}
 If $d\geq 8$ the general element $D\in\LL_{2,d}(3[s(2,d,n)],2^{h(2,d,n)})$ is irreducible.
\end{claim}
\begin{proof}[Proof of the Claim] Assume that the general element
  $D\in\LL_{2,d}(3[s(2,d,n)],2^{h(2,d,n)})$ is reducible. Then
  $D=D_1+D_2$, set 
$$d-1\geq b_i=\deg D_i>0$$
 and, by monodromy, $m_i:=\mul_{p_j}D_i$. Then $m_1+m_2=2$ so, up to order, there are two possibilities: either $m_1=0$ and $m_2=2$, or $m_1=m_2=1$. 

Assume first that $m_2=2$, that is we are assuming that $\LL_{2,b_2}(2^{h(2,d,n)})$ is not empty. We are assuming  $d\geq 8$, then $h(2,d,n)\geq 12$. Therefore,
by Theorem \ref{thm:AH}, this linear system is nonspecial and its virtual dimension
is 
		\[\binom{b_2+2}{2}-3h(2,d,n)-1=\binom{b_2+2}{2}-a(3,d)+s(2,d,n)-1.\]
A straightforward computation shows that this is non negative only for $d<8$.
  
Assume  $m_1=m_2=1$. Then there are plane curves $D_i$ of degree $b_i$ through $p_1,\dots,p_{h(2,d,n)}$, that is $D_i\in\LL_{2,b_i}(1^{h(2,d,n)})$. We may assume $b_1\leq b_2$, then the
virtual dimension of $\LL_{2,b_i}(1^{h(2,d,n)})$ is 
	\[\binom{b_i+2}{2}-h(2,d,n)-1,\]
and this is non negative only if 
		\[3b_i^2+9b_i\geq d^2+3d-22.\]
Therefore
		\[3b_1^2+9b_1\geq b_1^2+b_2^2+2b_1b_2+3b_1+3b_2-22\]
		\begin{equation}\label{eq:disuguaglianza_b}
		2b_1^2+6b_1\geq b_2^2+2b_1b_2+3b_2-22\ge b_1^2+2b_1^2+3b_1-22
		\end{equation}
		\[b_1^2-3b_1-22\leq 0\]
		so $b_1\le 6$. 
But then by (\ref{eq:disuguaglianza_b}) we have
		\[2b_1^2+6b_1\geq b_2^2+2b_1b_2+3b_2-22\ge b_2^2+2b_1^2+3b_2-22\]
		\[6b_1\geq b_2^2+3b_2-22\ge b_2^2+3b_1-22\]
		\[3b_1\geq b_2^2-22\]
		\[b_2^2\leq 3b_1+22\le 18+22=40\]
and we conclude $ b_2\le 6$.

Since $h(2,8,n)=12$ we have $b_i\geq 4$. Therefore the only possibilities are $b_1=4$ and $10\geq d\geq 8$, $b_1=5$ and $11\geq d\geq 10$ or $b_1=6$ and $d=12$.
A simple computation gives: $h(2,8,n)=12$, $h(2,9,n)=15$, $h(2,10,n)=19$, $h(2,11,n)=23$ and $h(2,12,n)=27$.
The divisor $D$ has a triple point  at $q$. Therefore the divisor $D_1$ has to
belong to one of the following linear systems:
\begin{itemize}
\item[\ $d=8$] $\LL_{2,4}(2,1^{12})$,
\item[\ $d=9$] $\LL_{2,4}(1^{15})$,
\item[$d=10$] either $\LL_{2,4}(1^{19})$ or $\LL_{2,5}(2,1^{19})$,
\item[$d=11$] $\LL_{2,5}(1^{23})$,
\item[$d=12$]  $\LL_{2,6}(2,1^{27})$.
\end{itemize}
An easy check shows that they are all empty.
\end{proof}

Set $d\geq 8$. Since the general element is irreducible the map
$\f_{2,d}$ is dominant.  
Assume that the geometric genus of $D$ is $0$ and, beside the singularities we imposed, there are  $l$ singular points of multiplicity $m_i$ and $t$ simple base points. Then we get
	\[0=\frac{(d-1)(d-2)}{2}-3-h(2,d,n)-\sum_{i=1}^{l}\frac{m_i(m_i-1)}{2}\]
				and
	\[1\leq d^2-9-4h(2,d,n)-s(2,d,n)-t-\sum_{i=1}^{l}m_i^2.\]
	This yields		
\begin{equation}
0\leq d^2-10-4h(2,d,n)-s(2,d,n)-t-\sum_{i=1}^{l}m_i^2-\left(\frac{(d-1)(d-2)}{2}-3-h(2,d,n)-\sum_{i=1}^{l}\frac{m_i(m_i-1)}{2}\right)\nonumber
\end{equation}
and, recall Equation~(\ref{eq:hsa}),
\begin{align}
	0 & \leq\frac{d^2+3d}{2}-8-3h(2,d,n)-s(2,d,n)-t-\sum_{i=1}^{l}\frac{m_i(m_i+1)}{2}\nonumber\\
	& =\frac{d^2+3d}{2}-8-a(2,d)-t-\sum_{i=1}^{l}\frac{m_i(m_i+1)}{2}\nonumber\\
				& =-t-\sum_{i=1}^{l}\frac{m_i(m_i+1)}{2},\nonumber
				\end{align}
therefore $l=0$ and $t=0$.
		
To conclude we compute the genus
		\begin{align}
		0 & =2g=(d-1)(d-2)-6-2h(2,d,n)=d^2-3d-4-2h(2,d,n)\nonumber\\
		& \ge d^2-3d-4-\frac{2}{3}a(2,d)
		=\frac{2}{3}(d^2 - 6d + 2).\nonumber
		\end{align}
This contradicts the assumption $d\geq 8$.

Assume $d=7$, then $h(2,7,n)=9$, $s(2,7,n)=0$. Let $C_1\in\LL_{2,3}(1,1^8)$ be a cubic through $q$ and $\{p_1,\ldots,p_8\}$, and $C_2\in\LL_{2,4}(2,1^7,2)$, a quartic singular in $q$ and $p_9$ and passing through $\{p_1,\ldots,p_8\}$. Then $D=C_1+C_2$ is a reducible element  of positive genus in $\LL_{2,7}(3,2^{9})$ and we conclude by Remark~\ref{rem:positive_sectional_genus}. 

Assume $d=6$, then $h(2,6,n)=6$, $s(2,6,n)=1$. Let 
$C_1\in \LL_{2,3}(2,1^6)$ and $C_2\in\LL_{2,3}(1[1],1^6)$ be
curves. Then  $D=C_1+C_2$ is a reducible element in
$\LL_{2,6}(3[1],2^{6})$. The curve $C_2$ is not rational therefore the sectional genus of $\LL_{2,6}(3[1],2^{6})$ is positive by Remark~\ref{rem:positive_sectional_genus}.

  For $d=4$ and $d=5$ it is an easy computation to see that the movable part of $\LL_{2,d}(3[s(2,d,n)],2^{h(2,d,n)})$ is given respectively by $\LL_{2,2}(1^3)$ and $\LL_{2,3}(2,1^4)$, which have genus 0.
			\end{proof}

For $d=4,5$ we bound the genus of $\LL_{3,d}(3[s(3,d,n)],2^{h(3,d,n)})$.
	\begin{pro}\label{pro:grado5}
		The sectional genus of
                $\LL_{3,5}(3[s(3,5,n)],2^{h(3,5,n)})$ is positive, for
                $n\geq 4$. 
		\begin{proof}
A direct computation gives $s(3,5,n)=2$, $h(3,5,n)=10$, $s(2,5,n)=0$, and $h(2,5,n)=4$.
Let $H+S\in \LL^H_{3}(5)$ be the unique divisor containing
$H$. By Remark~\ref{rem:positive_sectional_genus} it is enough to prove that $\LL_{3,5}(3[s(3,5,n)],2^{h(3,5,n)})_{|S}$ has positive sectional genus. 

The surface  
$$S\in\LL_{3,4}(2[2],2^6,1^4)(q[t_1,t_2],p_1,\ldots,p_6,z_1,\ldots,z_4)$$ 
is a quartic surface, the $p_i$'s are general and the $z_i$ are general on $H$.
By \cite[Theorem 4.1]{Me1}, $\LL_{3,4}(2^7)$ is nonspecial and the
general element has $7$
ordinary double points as unique singularities. The $z_i$'s
 are in general position on $H$ therefore the general element in
$$\LL_{3,4}(2^7,1^4)(q,p_1,\ldots,p_6,z_1,\ldots,z_4)$$
 has only $7$ ordinary double points. 
The linear system $\LL_{3,4}(3,2^5)(q,p_1,\ldots,p_5)$ is nonspecial of dimension 4 by Proposition
\ref{pro:3piudoppinonspeciale} and $\LL_{3,3}(2,2^5)$ is empty therefore 
$$\LL_{3,4}(3,2^5,1^5)(q,p_1,\ldots,p_5,p_6,z_1,\ldots,z_4)$$ 
 and henceforth $\LL_{3,4}(3,2^6,1^4)$ are empty.
This shows that for a general choice of $2$ tangent directions 
 the surface 
$$S\in\LL_{3,4}(2[2],2^6,1^4)(q[t_1,t_2],p_1,\ldots,p_6,z_1,\ldots,z_4)$$ has $7$
ordinary double points as unique singularities.
In particular $S$ is a, singular, $K3$ surface and it is not
uniruled. Therefore $\LL_{3,5}(3[s(3,5,n)],2^{h(3,5,n)})_{|S}$ has positive sectional genus. 
\end{proof}
\end{pro}

\begin{pro}\label{pro:grado4} The sectional genus of
  $\LL_{3,4}(3[s(3,4,n)],2^{h(3,4,n)})$ is positive, for $n\geq 4$.
\end{pro}
		\begin{proof}
Our choice of integers is $s(3,4,n)=1$, $h(3,4,n)=5$, $s(2,4,n)=0$,
and $h(2,4,n)=2$. Let $S+H$ be the unique element in
$\LL_3^H(4)$ containing the hyperplane $H$. Then 
$$S\in\LL_{3,3}(2[1],2^3,1^2)(q[t],p_1,p_2,p_3, p_4,p_5)$$ 
is a cubic surface with 4 double
points. 
It is easy to prove, with
reducible elements, that the scheme base
locus  of $\LL_3^H(4)$ is given by the assigned singularities and the lines spanned by $q$ and
$\{p_1,p_2,p_3\}$. Hence the fixed component of $\LL_3^H(4)_{|S}$ is given by the 3 lines spanned by $q$ and $p_1$, $p_2$, and $p_3$. Then the general element in the movable part of
$\LL_3^H(4)_{|S}$ has a triple point in $p_1$ and a double
point in $p_4$ and therefore positive genus. This is enough to conclude by Remark~\ref{rem:positive_sectional_genus}.
\end{proof}
We conclude the section collecting all the results we need.
\begin{pro}\label{pro:genus_bound} The sectional genus of
  $\LL_{3,d}(3[s(3,d,n)],2^{h(3,d,n)})$ is positive for  $n,d\geq4$.\end{pro}
\begin{proof} For $d=4,5$ this is the content of
  Propositions~\ref{pro:grado4},~\ref{pro:grado5}. For higher degrees
  observe that by construction  a general curve section of
  $\LL_{2,d}(3[s(2,d,n)],2^{h(2,d,n)})$ is an irreducible component of
  a curve section of
  $\LL_3^H(d)$ and therefore the sectional genus of
  $\LL_{3,d}(3[s(3,d,n)],2^{h(3,d,n)})$ is bounded by the sectional
  genus of $\LL_{2,d}(3[s(2,d,n)],2^{h(2,d,n)})$. Thus we conclude by Proposition~\ref{pro:degp2}.
\end{proof}

\section{Cubics and proof of Theorem~\ref{thm:Cremona-identifiability}}\label{sec:special}

Fix a pair $(n,d)$, with $n\geq d\geq 4$, a linear space $\Pi\cong\p^3$
and a point $q\in\Pi$.
Let $Z_0$ be the 0-dimensional scheme obtained as a limit of
$k(n,d)-1$ double points $n+1$ of which collapse to the point $q$ with
$s(3,d,n)$ tangent directions and $h(3,d,n)$ double points on
$\Pi$. Such a degeneration always exists since $s(3,d,n)\leq 4$.
Let $$\LL_0(n,d):=|\oo_{\p^n}(d)\otimes\I_{Z_0}|$$
be the associated linear system.

\begin{lem}\label{lem:degwithTisok}
If $n\ge d\ge 4$ and $(n,d)\neq (4,4)$, then the linear system $\LL_0(n,d)$ is non
special and its sectional genus is positive.
\end{lem}
\begin{proof}
Thanks to Proposition~\ref{pro:induction} for the nonspeciality we have only to worry about the tangent directions.
Let $T$ be the set of ${n+1\choose 2}$ tangent directions. 
Set $T_i:=\{y_1,\ldots,y_{i}\}\subset T$ be a subset of $i$ tangent
directions. Assume that $\LL_{n,d}(3[i],2^h)$ is nonspecial and
$\LL_{n,d}(3[i+1],2^h)$ is special for a general choice of $h$ double
points. Let $\f$ be the map associated to the linear system $\LL_{n,d}(3[i],2^{h-1})$. 

Since $T$ imposes independent conditions on cubics we may assume that
$y_{j}\not\in\Bs\LL_{n,d}(3[i],2^{h-1})$, for $j>i$. The speciality of
$\LL_{n,d}(3[i+1],2^{h})$ forces  $\f(y_{j})$ to be a vertex of
$\f(\p^n)$, for $j>i$. Hence a general divisor $D\in
\LL_{n,d}(3[i+1],2^{h})$ is singular at $y_{j}$, for $j>i$. By a monodromy argument then a general divisor in $\LL_{n,d}(3[{n+1\choose 2}],2^{h})$ is singular along $T$. Since $\LL_{n-1,3}(2^{{n+1\choose 2}})(T)$ is empty this yields 
$$\LL_{n,d}\left(3\left[{n+1\choose 2}\right],2^{h}\right)\subseteq \LL_{n,d}(4,2^{h})$$
and contradicts Proposition~\ref{pro:induction} ii) for $h\leq k(n,d)-n-2$.

We are left to bound the sectional genus of $\LL_0(n,d)$. 
Let $\tilde{T}_i$ be a set of ${n+1\choose 2}-i$ general tangent
directions. Let 
$$\LL(T_i):=\LL_{n,d}\left(3\left[{n+1\choose
  2}\right],2^{k(n,d)-n-2}\right)(q[T_i\cup\tilde{T}_i],p_1,\ldots,p_{k(n,d)-n-2}).$$
By definition we have $\LL(T_1)=\LL^H_n(d)$. 
Fix $D_1,\ldots,D_{n-3}\in \LL_n^H(d)$ general divisors containing
$\Pi$ and $Y_1,Y_2\in\LL_n^H(d)$ general divisors. Then,
  by Lemma~\ref{lem:firstdegeneration} (2), $\Pi$ is an
  irreducible component of $D_1\cdot\ldots\cdot D_{n-3}$ and, by
   Lemma~\ref{lem:firstdegeneration} (4), an irreducible component of $Y_1\cdot Y_2\cdot
  D_1\cdot\ldots\cdot D_{n-3}$ is a curve section of
  $\LL_{3,d}(3[s(3,d,n)],2^{h(3,d,n)})$. Hence, by Proposition~\ref{pro:genus_bound} and Remark~\ref{rem:positive_sectional_genus}, the claim is true for $i=1$.

 To conclude we increase $i$ recursively.  Fix $D_1,\ldots,D_{n-3}\in \LL(T_{i+1})$ general divisors containing
$\Pi$ and $Y_1,Y_2\in\LL(T_{i+1})$ general divisors.  
By construction $s(3,n,d)>0$ therefore we may assume that $\LL(T_{i+1})$ is a specialization of $\LL(T_i)$ moving a
 tangent direction in $\Pi$. In this degeneration all sections in $\LL(T_i)$
containing $\Pi$ are also sections of $\LL(T_{i+1})$. This shows that  $\Pi$ is an
  irreducible component of $D_1\cdot\ldots\cdot D_{n-3}$.
Next we may consider  $\LL(T_{i+1})$ as a specialization of
$\LL(T_i)$ moving a point outside $\Pi$. Via this degeneration we
prove that $\LL(T_{i+1})_{|\Pi}=\LL(T_i)_{|\Pi}$ and therefore an irreducible component of $Y_1\cdot Y_2\cdot
  D_1\cdot\ldots\cdot D_{n-3}$ is a curve section of
  $\LL_{3,d}(3[s(3,d,n)],2^{h(3,d,n)})$.
Then Proposition~\ref{pro:genus_bound} and Remark~\ref{rem:positive_sectional_genus} allow to conclude.
\end{proof}
	
\subsection{$(d=3)$}
The argument we used for forms of degree $d\geq 4$ does not work for cubics. Linear systems of cubics with a triple point and at least a  double point are always special. This forces us to apply a different strategy to study the degree of the map associated to $\LL_{n,3}(2^k)$. This is inspired by \cite{BO} and \cite{Pos} proof of Alexander--Hirschowitz Theorem.
Note that we are interested in integers $n$ such that 
$$k(n):=k(n,3)=\frac{{n+3\choose 3}}{n+1}$$
is an integer. This is equivalent to say that $n\equiv 0,1 (3)$. 
This property is preserved by codimension 3 linear spaces. This simple observation suggests the following induction procedure.

Assume $k(i)$ is an integer. Let $Z_1\subset\p^i$ be a 0-dimensional scheme of $k(i)-1$ general double points. Fix a general codimension 3 linear space $\Pi\subset\p^i$ and let $Z_0$ be a specialization of $Z_1$ with $k(i-3)-1$ double points on $\Pi$.
Therefore the linear system $\LL_{i,3}(2^{k(i)-1})$ specializes to a linear system $\LL_0$ and we may split $\LL_0$ as a direct sum of 
$$\tilde{\LL}\ {\rm and}\ \LL_{i-3,3}(2^{k(i-3)-1}), $$
where $\tilde{\LL}$ is the linear system of cubics containing $\Pi$ and singular in $i+1=k(i)-k(i-3)$ general points of $\p^i$  and in $k(i-3)$ general points of $\Pi$.
The linear system $\tilde{\LL}$ is known to be nonspecial by \cite[Proposition 5.4]{BO}, see also \cite[subsection 5.2]{Pos}, and $\LL_{i-3,3}(2^{k(i-3)-1})$ is nonspecial by Theorem~\ref{thm:AH}.
Let $g_i$ be the sectional genus of 
$\LL_{i,3}(2^{k(i)-1})$, for $i\equiv 0,1(3)$. Let $D_1, D_2, D_3$ be
three general cubics containing $\Pi$.  Considering the $\p^3$'s spanned by 4 double points, it is easy to check
that $\Pi$ is an irreducible component of $D_1\cdot D_2\cdot D_3$,
hence 
\begin{equation}
  \label{eq:deg3}
 g_i\geq g_{i-3}, 
\end{equation}
for $i\geq 6$. 


\begin{lem}
\label{ssec:3n=6} $g_6>0$. 
\end{lem}
\begin{proof}
The number $k(6)$ is 12.   Let
$Z_0=\{p_1,\ldots,p_8,z_1,z_2,z_3\}\subset\p^6$ be a specialization
with the $p_i$'s on a hyperplane $H$ and $z_j$ general. Then
$\LL_{6,3}(2^{11})$ specializes to a linear system
$\LL_{6,3}^H=\LL+\LL_{5,3}(2^8,1^3)$, with $\dim\LL_{6,3}^H=6$. It is
easy to see that $\dim\LL=1$ and  $\LL=H+\Lambda$ with $\Lambda$ a
pencil of quadrics of rank $4$ with vertex $\Span{z_1,z_2,z_3}$. Then
$M:=\Bs\Lambda$ is a cone over a normal elliptic curve in $\p^3$. In
particular $M$ is not rationally connected and therefore $\LL_{6,3}^H$ has positive sectional genus. Then, by Remark~\ref{rem:positive_sectional_genus}, we conclude that $g_6>0$. 
\end{proof}
\begin{lem}
  \label{ssec:3n=7} $g_7>0$.
\end{lem}
\begin{proof}
The number $k(7)$ is 15. Let
$Z_0=\{p_1,\ldots,p_{11},z_1,z_2,z_3\}\subset\p^7$ be a specialization
with the $p_i$'s on a hyperplane $H$ and $z_j$ general. Then
$\LL_{7,3}(2^{14})$ specializes to a linear system
$\LL_{7,3}^H=\LL+\LL_{6,3}(2^{11},1^3)$, with $\dim\LL_{7,3}^H=7$. It
is easy to see that $\dim\LL=3$ and  $\LL=H+\Lambda$ with $\Lambda$ a
linear system  of quadrics of rank $5$.  Then $M:=\Bs\Lambda$ is a union
of $16$ $\p^3$ meeting in $\Span{z_1,z_2,z_3}$.  
Let $\Pi_i=\Span{z_1,z_2,z_3,p_i}$, then
$\Pi_i\cap\Pi_j=\Span{z_1,z_2,z_3}$ and $\Pi_i\subset M$. By construction we have $\LL_{7,3|\Pi_i}^H\subset\LL_{3,3}(2^4)$ on the other hand specialization can only increase dimension of linear systems, therefore
$$\dim(\LL_{7,3}^H)_{|\Pi_i}\geq \dim\LL_{7,3}(2^{14})_{|\Pi_i}=4,$$
where the last equality is proved in \cite[subsection 5.4]{Pos}.
 Let $D_1,D_2\in\LL_{7,3}^H$ be two general
elements. Then $(D_1\cdot D_2)_{|\Pi_i}$ contains a twisted normal curve
passing through $\{z_1,z_2,z_3,p_i\}$. Set $Q_1,\ldots,Q_4\in\Lambda$ general
elements. Then the 1-cycle $Q_1\cap\ldots\cap Q_4\cap D_1\cap D_2$
contains rational curves intersecting in $\{z_1,z_2,z_3\}$ and it has
positive genus.
 This, by Remark~\ref{rem:positive_sectional_genus}, shows that $g_7>0$.
\end{proof}
We collected all needed result to prove Theorem~\ref{thm:Cremona-identifiability}.
\begin{proof}[Proof of Theorem~\ref{thm:Cremona-identifiability}] By
  \cite[Theorem 4.3, Proposition 2.4]{Me1} and \cite[Theorem 3.2]{AC} we may assume $d\leq n$ and
  $n\geq 4$. If  $d=2$ and $\dim\LL_{n,2}(2^h)=n$ then the map
  associated to  $\LL_{n,d}(2^h)$ is always of fiber type. 

If $d=3$, $n\equiv0(3)$, and $n\geq 6$ then Theorem~\ref{thm:AH} forces $h=k(n)-1$. Then, by Equation~(\ref{eq:deg3}) and
Lemma~\ref{ssec:3n=6}, the sectional genus of $\LL_{n,3}(2^{h})$ is positive. 
If $d=3$, $n\geq 7$ and $n\equiv1(3)$  we conclude as before via Equation~(\ref{eq:deg3}) and
Lemma~\ref{ssec:3n=7} that the sectional genus is positive. 
It is easy and well known that $\LL_{4,3}(2^6)$
induces a fiber type map that contracts the rational normal curves through the 6 points.
This analysis proves the theorem for $d\leq 3$.


Assume that $n\geq d\geq 4$. By
Theorem~\ref{thm:AH}, $h=k(n,d)-1$. If $n=d=4$, then $h=13$. There is a pencil of quadrics in $\p^4$ through 13 general points, so $\LL_{4,4}(2^{13})$ admits a linear subsystem of reducible divisors with base locus in codimension 2, hence the associated map cannot be birational. Suppose then $(n,d)\neq (4,4)$. By Lemma~\ref{lem:degwithTisok}, $\LL_0(n,d)$ is a specialization of $\LL_{n,d}(2^h)$  and it has positive sectional genus. This shows that $\LL_{n,d}(2^h)$ does not define a Cremona modification.
\end{proof}

\end{document}